\documentclass[12pt,reqno]{amsart}

\usepackage{mathabx}
\DeclareMathAlphabet{\mathsl}{OT1}{cmss}{m}{sl}
\SetMathAlphabet{\mathsl}{bold}{OT1}{cmss}{bx}{sl}

\DeclareFontEncoding{LS1}{}{}
\DeclareFontSubstitution{LS1}{stix}{m}{n}
\DeclareMathAlphabet{\mathscr}{LS1}{stixscr}{m}{n} 

\usepackage{geometry}

\usepackage{currfile,datetime}
\usepackage[sort,compress]{cite}

\usepackage{mathrsfs}

\usepackage{amsthm,amsmath,amssymb,bbm,geometry,epsfig,listings,
	paralist,enumerate,comment,nicefrac,verbatim,multirow,xcolor,wasysym,tikz}
\usepackage[pagewise]{lineno}[4.41]
\usepackage{marginnote}\reversemarginpar

\usepackage[colorlinks,linkcolor=red,citecolor=red]{hyperref}

\usepackage[capitalise,compress]{cleveref} 

\crefname{equation}{}{}

\newtheorem{lemma}{Lemma}[section]
\newtheorem{remark}[lemma]{Remark}

\newtheorem{theorem}[lemma]{Theorem}

\newtheorem{setting}[lemma]{Setting}

\crefname{subsection}{Subsection}{Subsections}
\crefname{enumi}{item}{items}
\creflabelformat{enumi}{(#2\textup{#1}#3)}

\providecommand{\N}{{\ensuremath{\mathbbm{N}}}}
\providecommand{\Z}{{\ensuremath{\mathbbm{Z}}}}
\providecommand{\R}{{\ensuremath{\mathbbm{R}}}}

\renewcommand{\P}{{\ensuremath{\mathbbm{P}}}}
\providecommand{\E}{{\ensuremath{\mathbbm{E}}}}

\providecommand{\bfD}{{\ensuremath{\mathbf{D}}}}
\providecommand{\calR}{{\ensuremath{\mathcal{R}}}}
\providecommand{\calD}{{\ensuremath{\mathcal{D}}}}
\providecommand{\bfN}{{\ensuremath{\mathbf{N}}}}

\newcommand{\calP}{\mathcal{P}}
\newcommand{\xeqref}[1]{}

\newcommand{\supnorm}[1]{{\left\vert\kern-0.25ex\left\vert\kern-0.25ex\left\vert #1 
    \right\vert\kern-0.25ex\right\vert\kern-0.25ex\right\vert}}
\renewcommand{\gets}{\curvearrowleft}

\title[]{Rectified deep neural networks overcome the curse of dimensionality\\ when approximating solutions of \\McKean--Vlasov stochastic differential equations}

\author[A. Neufeld]{Ariel Neufeld$^{1}$}
\address{$^1$  Division of Mathematical Sciences, School of Physical and Mathematical Sciences, Nanyang Technological University, Singapore}
\email{ariel.neufeld@ntu.edu.sg}

\author[T.A. Nguyen]{Tuan Anh Nguyen$^{2}$}
\address{$^2$  Faculty of Mathematics, Bielefeld University, Bielefeld, Germany}
\email{tnguyen@math.uni-bielefeld.de}

\thanks{
Financial support by the Nanyang Assistant Professorship Grant (NAP Grant) \emph{Machine Learning based Algorithms in
Finance and Insurance} is gratefully acknowledged.}



\allowdisplaybreaks
\begin{document}

\begin{abstract}
In this paper we prove that rectified deep neural networks do not suffer from the curse of dimensionality when approximating McKean--Vlasov SDEs in the sense that 
the number of parameters in the deep neural networks only  grows polynomially in the space dimension $d$ of the SDE and the reciprocal
 of the accuracy $\epsilon$.
\end{abstract}
	\maketitle
\section{Introduction}

In probability theory, a McKean--Vlasov 
stochastic differential equation (SDE)
is an SDE where the coefficients of the diffusion depend on the distribution of the solution itself.
These equations are a model for the Vlasov equation and were first studied by McKean~\cite{HPMcK1966}. With
McKean--Vlasov SDEs we can have a stochastic representation
of solutions of nonlinear, possibly non-local parabolic partial  differential equations (PDEs), e.g.,  Vlasov’s equation, Boltzmann’s equation, or Burgers’ equation.
 Moreover, it is known that weakly dependent
diffusion processes converge to independent solutions of McKean--Vlasov SDEs as the system size tends
to infinity. This phenomenon was called \emph{propagation of chaos} by Kac 
\cite{Kac1956}
 and is well studied
in the literature, see, e.g., \cite{Szn1991}. 
Numerical approximation of these McKean--Vlasov SDEs is an important topic in computational mathematics, see, e.g.,
\cite{RS2022,KNRS2022,BST2019,BH2022,dRES2022,dRST2023,BRRS2023,chen2023wellposedness,leobacher2022well,reisinger2021fast,bao2021first,agarwal2023numerical,chen2022euler,biswas2022explicit,chen2022flexible,reisinger2020regularity,reisinger2020path,reisinger2020posteriori,reisinger2020optimal,cass2019support,szpruch2019iterative,szpruch2019antithetic,alrachid2019new}. 

We are interested in the numerical approximation of \emph{high-dimensional} McKean--Vlasov SDEs, more precisely, in the following question: which approximation methods do not suffer from the curse of dimensionality 
in the sense that the computational complexity only grows polynomially in the space dimension $d$ of the SDE and the reciprocal
$\frac{1}{\epsilon}$ of the accuracy $\epsilon$?
To the best of our knowledge there has been only a few papers which have discussed McKean--Vlasov SDEs in high dimensions. First, \cite{HKN2022} has proven that multilevel Picard approximations do not suffer the curse of dimensionality when approximately solving McKean--Vlasov SDEs. 
Moreover, \cite{GMW2022,HHL2023,PW2022} developed deep learning based algorithms to approximately solve high-dimensional McKean--Vlasov equations. Furthermore, \cite{HHL2023} provided a priori error estimates under a suitable metric between the McKean--Vlasov equations under consideration and their standard forward backward SDE approximations which are used in their algorithm that can be controlled to be arbitrarily small and are dimension-free. However, it is left open whether the number of parameters in the deep neural networks used to approximate these standard forward backward SDEs suffers from the curse of dimensionality, or not.

In this paper we  show that rectified deep neural networks (DNNs) do not suffer from the curse of dimensionality when approximating McKean--Vlasov SDEs, either, in the sense that
the number of parameters in the deep neural networks only  grows polynomially in the space dimension $d$ of the SDE and the reciprocal
 of the accuracy~$\epsilon$.

\subsection{Notations}\label{s01}
Throughout this paper we use the following notations.
Let $\R$ denote the set of all real numbers. Let $\Z$ denote the set of all integers. Let $\N$ denote the set of all natural numbers starting at $1$ and let $\N_0=\N\cup\{0\}$. 
Let $\left \lVert\cdot \right\rVert, \supnorm{\cdot} \colon (\cup_{d\in \N} \R^d) \to [0,\infty)$, 
$\dim \colon (\cup_{d\in \N}\R^d) \to \N$
satisfy for all $d\in \N$, $x=(x_1,\ldots,x_d)\in \R^d$ that $\lVert x\rVert=\sqrt{\sum_{i=1}^d(x_i)^2}$, $\supnorm{x}=\max_{i\in [1,d]\cap \N}|x_i|$, and
$\dim (x)=d$.
For every probability space $(\Omega,\mathcal{F}, \P)$,  $d\in \N$,
$p\in [1,\infty)$, and every random variable $X\colon \Omega\to\R^d$ let $\lVert X\rVert_{L^p(\P;\R^d)} \in [0,\infty]$ satisfy that
$\lVert X\rVert_{L^p(\P;\R^d)}^p= \E[\lVert X\rVert^p]$.

\subsection{A mathematical framework for DNNs}

First of all in \cref{m07} below we introduce a mathematical framework for DNNs with ReLU activation function.

\begin{setting}[A mathematical framework for DNNs]\label{m07}
Let
 $\mathbf{A}_{d}\colon \R^d\to\R^d $, $d\in \N$, satisfy for all $d\in\N$, $x=(x_1,\ldots,x_d)\in \R^d$ that 
\begin{align}
\mathbf{A}_{d}(x)= \left(\max\{x_1,0\},\max\{x_2,0\},\ldots,\max\{x_d,0\}\right).
\end{align}
Let $\mathbf{D}=\cup_{H\in \N} \N^{H+2}$. Let
\begin{align}
\mathbf{N}= \bigcup_{H\in  \N}\bigcup_{(k_0,k_1,\ldots,k_{H+1})\in \N^{H+2}}
\left[ \prod_{n=1}^{H+1} \left(\R^{k_{n}\times k_{n-1}} \times\R^{k_{n}}\right)\right].
\end{align} Let $\mathcal{D}\colon \mathbf{N}\to\mathbf{D}$, 
$\mathcal{P}\colon \mathbf{N}\to \N$,
$
\mathcal{R}\colon \mathbf{N}\to (\cup_{k,l\in \N} C(\R^k,\R^l))$
satisfy that
for all $H\in \N$, $k_0,k_1,\ldots,k_H,k_{H+1}\in \N$,
$
\Phi = ((W_1,B_1),\ldots,(W_{H+1},B_{H+1}))\in \prod_{n=1}^{H+1} \left(\R^{k_n\times k_{n-1}} \times\R^{k_n}\right), 
$
$x_0 \in \R^{k_0},\ldots,x_{H}\in \R^{k_{H}}$ with the property that
$\forall\, n\in \N\cap [1,H]\colon x_n = \mathbf{A}_{k_n}(W_n x_{n-1}+B_n )
$ we have that
\begin{align}
\mathcal{P}(\Phi)=\sum_{n=1}^{H+1}k_n(k_{n-1}+1),
\quad 
\mathcal{D}(\Phi)= (k_0,k_1,\ldots,k_{H}, k_{H+1}),
\end{align}
$
\mathcal{R}(\Phi )\in C(\R^{k_0},\R ^ {k_{H+1}}),
$
and
\begin{align}
 (\mathcal{R}(\Phi)) (x_0) = W_{H+1}x_{H}+B_{H+1}.
\end{align}
\end{setting}
Let us comment on the mathematical objects  in  \cref{m07}. For all $ d\in \N $, $ \mathbf{A}_d\colon\R^d\to\R^d$  refers to the componentwise  rectified linear unit (ReLU) activation function. 
Moreover, by $ \mathbf{N} $ we denote the set of all
parameters characterizing artificial feed-forward DNNs. 
In addition, by $ \calD $ we denote the function that maps the parameters characterizing a DNN to the vector of its layer dimensions. Furthermore, by $ \calR $   we denote the operator that maps each parameters characterizing a DNN to its corresponding function. Finally, by $ \calP $ we denote the function that counts the number of parameters of the corresponding DNN.

\subsection{Main result}
\begin{theorem}\label{g01b}
 Assume \cref{m07}. Let $T\in(0,\infty)$,
$c\in [1,\infty)$,
$r\in\N$. For every
$d\in \N$, $\varepsilon\in [0,1)$ let $\mu_{d,\varepsilon}\in C(\R^d\times\R^d,\R^d)$,
$f_{d,\varepsilon}\in C(\R^d,\R)$.
For every
$d\in \N$, $\varepsilon\in (0,1)$ let
$\Phi_{\mu_{d,\varepsilon}},\Phi_{f_{d,\varepsilon}}\in \bfN$
satisfy that
$\mu_{d,\varepsilon}=\calR(\Phi_{\mu_{d,\varepsilon}}) $ and
$f_{d,\varepsilon}=\calR(\Phi_{f_{d,\varepsilon}}) $. Moreover,
 assume for all 
$d\in \N$,
$\varepsilon\in (0,1)$, $x,y,x_1,x_2,y_1,y_2\in \R^d$  that
\begin{align}\label{g04}
\left\lVert
\mu_{d,\varepsilon}(x_1,y_1)-
\mu_{d,\varepsilon}(x_2,y_2)\right\rVert\leq 0.5c\lVert x_1-y_1\rVert+0.5c \lVert x_2-y_2\rVert,
\end{align}
\begin{align}\label{g06}
\left\lVert
\mu_{d,\varepsilon}(x_1,y_1)-
\mu_{d,0}(x_1,y_1)\right\rVert\leq cd^c\varepsilon+ 0.5cd^c\varepsilon\lVert x_1\rVert^r+
0.5cd^c\varepsilon\lVert y_1\rVert^r,
\end{align}
\begin{align}\label{g11}
\left \lvert f_{d,\varepsilon}( x )-f_{d,\varepsilon}(y)\right\rvert\leq c\lVert x-y\rVert,
\end{align}
\begin{align}\label{g07}
\left\lvert
f_{d,\varepsilon}(x)-f_{d,0}(x)
\right\rvert\leq \varepsilon(1+ \lVert x\rVert^r),
\end{align}
\begin{align}\label{g08}
\left\lvert
f_{d,\varepsilon}(0)
\right\rvert+
\left\lVert\mu_{d,\varepsilon}(0,0)\right\rVert
\leq cd^c,
\end{align}
and
\begin{align}\label{g09}
\max\!\left\{\calP(\Phi_{\mu_{d,\varepsilon}}),
\calP(\Phi_{f_{d,\varepsilon}})\right\}\leq d^c\varepsilon^{-c}.
\end{align}
Then the following items hold.
\begin{enumerate}[i)]
\item There exists a probability space $(\Omega,\mathcal {F}, \P)$, a family of standard Brownian motions $W^d\colon [0,T]\times \Omega\to \R^d$, $d\in \N$, and  $(X^{d,x}(t))_{d\in\N,x\in\R^d,t\in [0,T]}$, $d\in \N$, such that for every $d\in \N$, $ (X^{d,x}(t))_{t\in [0,T]} $ 
is
$(\sigma( (W_s)_{s\in[0,t]} ))_{t\in[0,T]}$-adapted stochastic process with continuous sample paths
  and satisfies for all $t\in [0,T]$, $m\in [1,\infty)$ that
$
\E\!\left[
\sup_{s\in [0,T]}\left\lvert
X^{d,x}(s)\right\rvert^m\right]<\infty
$ and
$\P$-a.s.
\begin{align} \begin{split} X^{d,x}(t)=x+\int_{0}^{t}\E \!\left[\mu_{d,0}\bigl(y,X^{d,x}(s)\bigr)
\right]\bigr|_{y=X^{d,x}(s)}\,ds+W^{d}(t).\end{split}\label{g10}
\end{align}
\item There exist $\eta\in (0,\infty)$ and $(\Psi_{d,\epsilon}) _{d\in\N,\epsilon\in \N}\subseteq\bfN$ 
such that for all $d\in \N$, $\epsilon\in (0,1)$ we have that
$
\calP(\Psi_{d,\epsilon})\leq\eta d^{\eta}\epsilon^{-\eta}
$ and
\begin{align}
\left(
\int_{[0,1]^d}
\left\lvert
(\calR(\Psi_{d,\epsilon}))(x)
-\E\!\left[f_{d,0}(X^{d,x}(T))\right]\right\rvert^2 dx
\right)^\frac{1}{2}
<\epsilon.
\end{align}
\end{enumerate}

\end{theorem}

\begin{remark}

For every $d\in \N$ let $P_2(\R^d)$ denote the set of probability measure on $\R^d$ with finite second moments 
and
let $b_{d,0}\colon \R^d\times {P}_2(\R^d)\to \R^d$ satisfy for all $x\in\R^d$, $\nu\in P_2(\R^d)$ that
$
b_{d,0}(x,\nu)=\int\mu_{d,0} (x,y)\,\nu(dy).
$
Then we have that
\begin{align}
X^{d,x}(t)
&=x+\int_{0}^{t}\E \!\left[\mu_{d,0}\bigl(y,X^{d, x}(s)\bigr)
\right]\bigr|_{y=X^{d, x}(s)}\,ds+W^{d,\theta}(t)\nonumber\\
&=x+\int_{0}^{t}b_{d,0}(X^{d,x}(s),\P_{X^{d, x}(s)})\,ds+W^{d,\theta}(t).
\end{align}
\end{remark}
Let us comment on the mathematical objects in \cref{g01b}.
For every $d\in \N$, $x\in \R^d$ we want to approximate $\E\!\left[f_{d,0}(X^{d,x}(T))\right]$ where $f_{d,0}\in C(\R^d,\R)$ and $X^{d,x}(T)$ is the terminal condition of the corresponding McKean--Vlasov SDE in \eqref{g10}. 
Note that in the non-McKean--Vlasov setting, 
$\E\!\left[f_{d,0}(X^{d,x}(T))\right]$, $d\in \N$, are solutions of the associated linear Kolmogorov PDEs.
Conditions \eqref{g06}, \eqref{g07}, and \eqref{g09}
state that the functions $\mu_{d,0}$, $f_{d,0}$, $d\in \N$, can be approximated by  functions $\mu_{d,\varepsilon}$, $f_{d,\varepsilon}$, $d\in \N$, $\varepsilon\in (0,1)$, which are generated by  DNNs, without the curse of dimensionality. Conditions \eqref{g04} and \eqref{g11} are Lipschitz-type conditions. 
Condition \eqref{g08} states that the initial values of the functions $\mu_{d,\varepsilon}$ and $f_{d,\varepsilon}$, $\varepsilon\in (0,1)$, $d\in \N$,
grow polynomially in the dimension $d\in \N$.
Note that \eqref{g04} is  needed for the existence and uniqueness of \eqref{g10}, which we will discuss later in the proof of the main theorem.  Under these assumptions \cref{g01b} states that, roughly speaking, if DNNs can approximate the drift coefficients $\mu_{d,0}$
and the functions $f_{d,0}$, 
 $d\in \N$, without the curse of dimensionality, then they can also approximate the expectations
$\E\!\left[f_{d,0}(X^{d,x}(T))\right]$
 without the curse of dimensionality.
We also refer to \cite{AJK+2023,CR2023,CHW2022,GHJvW2023,HJKN2020a,JSW2021,NNW2023,GS2021a,GS2021,JKvW2023} for similar results  that DNNs can overcome the curse of dimensionality when approximating PDEs and  SDEs (but not of McKean--Vlasov type).

\subsection{Outline of the proof and organization of our paper}
\cref{g01b} follows directly from \cref{g01} in \cref{s02}, see the proof of \cref{g01b} after the proof of \cref{g01}. To prove \cref{g01} we need \cref{a01,k12,s14} in \cref{s03,s04}. 
The main idea of the proof are
the fact that 
multilevel Picard approximations (see \eqref{a05b}) are proven (cf. \cite[Theorem~1.1]{HKN2022}) to overcome the curse of dimensionality when approximating the 
solutions of \eqref{g10}, $d\in \N$, together with
 the fact that 
when the corresponding  drift coefficients $\mu_{d,0}$, $d\in \N$, and the functions $f_{d,0}$, $d\in \N$, can be represented by DNNs, then the corresponding multilevel Picard  and  Monte Carlo approximations can also be represented by DNNs (cf. \cref{k12,s14}).
To prove \cref{k12,s14} we need some basic DNN calculus, especially, the fact that the composition and the sum of functions generated by DNNs with the same length are again a function generated by DNNs (cf. \cref{m11b,b01b}) and the fact that the identity can be represented by DNNs (cf. \cref{b03}).
 Since $\mu_{d,0}$ and $f_{d,0}$, $d\in \N$, are not supposed to be represented by DNNs but to be approximated by them, we need a perturbation lemma, \cref{a01}. Combining \cref{a01} with the error estimate for the corresponding multilevel Picard and the Monte Carlo approximations we obtain the error estimate for the multilevel Picard and the Monte Carlo approximations with coefficient functions generated by DNNs (cf. \eqref{a76} in the proof of \cref{g01}).

\section{Perturbation lemma}\label{s03}
First of all, we discuss existence and uniqueness of the McKean--Vlasov SDE~\eqref{g10} in \cref{a01b} below. Afterwards, we formulate and prove \cref{a01}, which is the perturbation lemma mentioned in the outline of the proof.
\begin{lemma}[Existence and uniqueness]\label{a01b}
 Let $T\in(0,\infty)$,
$c\in [1,\infty)$,
$d\in \N$, 
$\mu_{}\in C(\R^d\times\R^d,\R^d)$
satisfy 
  for all 
$x_1,x_2,y_1,y_2\in \R^d$  that
\begin{align}
\left\lVert
\mu(x_1,y_1)-
\mu(x_2,y_2)\right\rVert\leq 0.5c\lVert x_1-y_1\rVert+0.5c \lVert x_2-y_2\rVert.\label{c02}
\end{align}
Let $(\Omega,\mathcal{F},\P)$ be a  probability space. Let $W=(W(t))_{t\in [0,T]}\colon [0,T]\times\Omega\to\R^d$ be a standard
$d$-dimensional Brownian motion with continuous sample paths. Then for every $x\in \R^d$ there exists
a 
$(\sigma( (W(s))_{s\in[0,t]} ))_{t\in[0,T]}$-adapted stochastic process  $X^{x}=(X^{x}(t))_{t\in [0,T]}\colon [0,T]\times \Omega\to \R^d$  with continuous sample paths 
satisfying for all $t\in [0,T]$, $m\in [1,\infty)$ that
$
\E\!\left[
\sup_{s\in [0,T]}\left\lvert
X^{x}(s)\right\rvert^m\right]<\infty
$ and
\begin{align}
X^{x}(t)=x+\int_{0}^{t}
\E\!\left[\mu(y,X^{x}(s))\right]\bigr|_{y=X^{x}(s)}\,ds +W(t).
\end{align}
\end{lemma}
\begin{proof}
[Proof of \cref{a01b}]
Let $P_2(\R^d)$ denote the set of probability measure on $\R^d$ with finite second moments. Let $\mathbf{W}_2\colon P_2(\R^d)\times P_2(\R^d)\to [0,\infty)$ denote the Wasserstein $2$-distance. Recall that for $\nu_1,\nu_2\in P_2(\R^d)$ the Wasserstein $2$-distance between $\nu_1$ and $\nu_2$ is defined as
\begin{align}
\mathbf{W}_2(\nu_1,\nu_2)=
\inf\!\left\{\left(\hat{\E} \!\left[\lVert Y_1-Y_2 \rVert^2\right]\right)^{1/2}\colon
\begin{aligned}
(\hat{\Omega},\hat{\mathcal{F}}, \hat{\P}) \text{ is a probability space,}\\
Y_1\in L^2(\hat{\P},\R^d),Y_2\in L^2(\hat{\P},\R^d),\\
\hat{\P}\circ Y_1^{-1}=\nu_1,\hat{\P}\circ Y_2^{-1}=\nu_2
\end{aligned}
\right\}.	
\end{align}
Let $b\colon \R^d\times {P}_2(\R^d)\to \R^d$ satisfy for all $x\in\R^d$, $\nu\in P_2(\R^d)$ that
\begin{align}\label{c03}
b(x,\nu)=\int\mu (x,y)\,\nu(dy).
\end{align}
Then \eqref{c02} shows that for all $x\in\R^d$, all probability space $(\hat{\Omega},\hat{\mathcal{F}},\hat{\P})$, and all random variables $Y_1,Y_2\colon \hat{\Omega}\to \R^d$
with the property that
$\hat{\P}(Y_1\in \cdot )=\nu_1$,
$\hat{\P}(Y_2\in \cdot )=\nu_2$
we have that
\begin{align} 
\left\lVert
b(x,\nu_1)-b(x,\nu_2)\right\rVert&=
\left\lVert
\int\mu (x,y)\,\nu_1(dy)
-
\int\mu (x,y)\,\nu_2(dy)
\right\rVert\nonumber\\&
=\left\lVert\hat{\E}\!\left[ \mu(x,Y_1)\right]-
\hat{\E}\!\left[ \mu(x,Y_2)\right]\right\rVert\nonumber\\
&\leq 0.5c \lVert Y_1-Y_2\rVert_{L^2(\hat{\P},\R^d)}.\label{c04}
\end{align}
This and the definition of the Wasserstein $2$-distance show for all $x\in\R^d$, $\nu_1,\nu_2\in P_2(\R^d)$ that 
\begin{align}
\left\lvert
b(x,\nu_1)-b(x,\nu_2)\right\rvert\leq 0.5c\mathbf{W}_2(\nu_1,\nu_2).
\end{align}
Furthermore, \eqref{c03}, \eqref{c02}, and the Cauchy-Schwarz inequality show for all $x_1 ,x_2\in \R^d$, $\nu\in P_2(\R^d)$ that
\begin{align}
\left\lVert
b(x_1,\nu)-b(x_2,\nu)
\right\rVert&= 
\left\lVert
\int \mu (x_1,y)-\mu(x_2,y)\,\nu(dy)
\right\rVert\nonumber\\&\leq 
\int\left\lVert \mu (x_1,y)-\mu(x_2,y)\right\rVert\nu(dy)
\nonumber\\
&\leq 0.5c \left\lVert x_1-x_2\right\rVert 
\end{align}
and\begin{align}
\left\langle
x_1-x_2,
b(x_1,\nu)-b(x_2,\nu)\right\rangle
&
\leq \left\lVert x_1-x_2\right\rVert\left\lVert
b(x_1,\nu)-b(x_2,\nu)
\right\rVert\leq 0.5c \left\lVert x_1-x_2\right\rVert ^2,
\end{align}
where $\langle\cdot,\cdot\rangle$ denotes the standard scalar product on $\R^d$.
This, \eqref{c04}, and a result on existence and uniqueness of McKean--Vlasov SDEs (see, e.g.,
\cite[Theorem 3.3]{RST2019}) show that 
for every  $x\in \R^d$
there exists a unique 
$(\sigma( (W(s))_{s\in[0,t]} ))_{t\in[0,T]}$-adapted 
stochastic process 
$(X^{x}(t))_{t\in [0,T]}$
such that
for all $t\in [0,T]$
we have that $\P$-a.s.\ 
\begin{align}  X^{x}(t)&=x+\int_{0}^{t}b(X^{x}(s),\P_{X^{x}(s)})\,ds+W(t)\nonumber\\&=
x+\int_{0}^{t}\int \mu(X^{x}(s),y)
\P_{X^{x}(s)}(dy)
\,ds+W(t)\nonumber\\
&=x+\int_{0}^{t}\E \!\left[\mu\bigl(y,X^{x}(s)\bigr)
\right]\bigr|_{y=X^{x}(s)}\,ds+W(t)
\end{align}
and 
such that for all $m\in [1,\infty)$ we have that
$
\E\!\left[
\sup_{s\in [0,T]}\left\lvert
X^{x}(s)\right\rvert^m\right]<\infty.
$ This completes the proof of \cref{a01b}.
\end{proof}
\begin{lemma}[Perturbation lemma]\label{a01}
Let $T\in(0,\infty)$,
$b,c\in [1,\infty)$,
$d,p,r\in\N$.  For every $\varepsilon\in [0,1)$ let $\mu_\varepsilon\in C(\R^d\times\R^d,\R^d)$,
$f_\varepsilon\in C(\R^d,\R)$.
Assume for all $\varepsilon\in [0,1)$, $x,y,x_1,x_2,y_1,y_2\in \R^d$  that
\begin{align}\label{a02}
\left\lVert
\mu_\varepsilon(x_1,y_1)-
\mu_\varepsilon(x_2,y_2)\right\rVert\leq 0.5c\lVert x_1-y_1\rVert+0.5c \lVert x_2-y_2\rVert,
\end{align}
\begin{align}\label{a03}
\left\lVert
\mu_\varepsilon(x_1,y_1)-
\mu_0(x_1,y_1)\right\rVert\leq b\varepsilon+ 0.5b\varepsilon\lVert x_1\rVert^r+
0.5b\varepsilon\lVert y_1\rVert^r,
\end{align}
\begin{align}
\left\lvert
f_\varepsilon(x)-f_0(x)
\right\rvert\leq \varepsilon(1+ \lVert x\rVert^r),\quad \left \lvert f_0( x )-f_0(y)\right\rvert\leq c\lVert x-y\rVert.\label{a04b}
\end{align}
Let $(\Omega,\mathcal{F},\P)$ be a  probability space. Let $W=(W(t))_{t\in [0,T]}\colon [0,T]\times\Omega\to\R^d$ be a standard
$d$-dimensional Brownian motion with continuous sample paths. For every 
$x\in \R^d$,
$\varepsilon\in [0,1)$ let $X^{\varepsilon,x}=(X^{\varepsilon,x}(t))_{t\in [0,T]}\colon [0,T]\times \Omega\to \R^d$ be a 
$(\sigma( (W(s))_{s\in[0,t]} ))_{t\in[0,T]}$-adapted stochastic process with continuous sample paths satisfying for all $t\in [0,T]$, $m\in [1,\infty)$ that
$
\E\!\left[
\sup_{s\in [0,T]}\left\lvert
X^{x}(s)\right\rvert^m\right]<\infty
$  and
\begin{align}\label{a04}
X^{\varepsilon,x}(t)=x+\int_{0}^{t}
\E\!\left[\mu_\varepsilon(y,X^{\varepsilon,x}(s))\right]\bigr|_{y=X^{\varepsilon,x}(s)}\,ds +W(t)
\end{align}
(for existence and uniqueness cf. \cref{a01b}).
Then the following items hold.
\begin{enumerate}[(i)]
 \item \label{a07b} For all
$\varepsilon\in [0,1)$, 
$t\in [0,T]$,
$x\in\R^d$ we have that 
\begin{align}
\left\lVert 
X^{\varepsilon,x}(t)\right\rVert_{L^{pr}(\P;\R^d)}\leq
\left(
\lVert x\rVert+T\left\lVert\mu_\varepsilon(0,0)\right\rVert+\sqrt{T(d+2pr)}\right)e^{ct}.
\end{align}
\item \label{a11}For all
$\varepsilon\in [0,1)$, 
$t\in [0,T]$,
$x\in\R^d$ we have that
\begin{align} \begin{split} 
&
\left\lVert
X^{\varepsilon,x}(t)-X^{0,x}(t)\right\rVert_{L^p(\P;\R^d)}\leq 
Tb\varepsilon \left(1+
\lVert x\rVert+T\left\lVert\mu_0(0,0)\right\rVert+\sqrt{T(d+2pr)}\right)^re^{(r+1)cT}.
\end{split}\end{align}
\item \label{c01}
For all
$\varepsilon\in [0,1)$, 
$t\in [0,T]$,
$x\in\R^d$ we have that
\begin{align} \begin{split} 
&\left\lVert f_\varepsilon(X^{\varepsilon,x}(t))
-
f_0(X^{0,x}(t))\right\rVert_{L^p(\P;\R)}
\\
&\leq 
b\varepsilon \left(1+
\lVert x\rVert+T
\max\{
\left\lVert\mu_\varepsilon(0,0)\right\rVert,
\left\lVert\mu_0(0,0)\right\rVert
\}
+\sqrt{T(d+2pr)}\right)^re^{(r+2)cT}
.
\end{split}
\end{align}
\end{enumerate}

\end{lemma}
\begin{proof}[Proof of \cref{a01}]
First, note for all $t\in (0,T]$ that $\left\lVert \frac{W(t)}{\sqrt{t}}\right\rVert^2$ is chi-squared distributed with $d$ degrees of freedom. This and Jensen's inequality show for all $t\in (0,T]$
 that
\begin{align}
\left(
\E \!\left[
\left\lVert W(t)\right\rVert^{{pr}}
\right]\right)^2\leq 
\E \!\left[
\left\lVert W(t)\right\rVert^{2pr}
\right]=(2t)^{{pr}}\frac{\Gamma(0.5d+{pr})}{\Gamma(0.5d)}
=(2t)^{{pr}}\prod_{k=0}^{{pr}-1}(0.5d+k)
\end{align} and
\begin{align}
\left\lVert
W(t)
\right\rVert_{L^{{pr}}(\P;\R^d)}
&=\left(
\E \!\left[
\left\lVert W(t)\right\rVert^{{pr}}
\right]\right)^\frac{2}{2pr}
\leq (2t)^\frac{1}{2}
\left(\prod_{k=0}^{{pr}-1}(0.5d+k)\right)^\frac{1}{2pr}
\nonumber
\\&\leq \sqrt{2t\left(\frac{d}{2}+{pr}-1\right)}\leq \sqrt{t(d+2pr)}.\label{a05}
\end{align}
Thus, the triangle inequality, Jensen's inequality, the disintegration theorem, and \eqref{a02} imply for all $\varepsilon\in [0,1)$, 
$t\in[0,T]$,
$x\in \R^d$ that
\begin{align} 
\left\lVert 
X^{\varepsilon,x}(t)\right\rVert_{L^{pr}(\P;\R^d)}&\leq 
\lVert x\rVert+\int_{0}^{t}
\left\lVert
\E\!\left[\mu_\varepsilon(y,X^{\varepsilon,x}(s))\right]\bigr|_{y=X^{\varepsilon,x}(s)}
\right\rVert_{L^{pr}(\P;\R^d)}
ds +\left\lVert W(t)\right\rVert_{L^{pr}(\P;\R^d)}\nonumber\\
&\leq 
 \lVert x\rVert+\int_{0}^{t}
\left\lVert
\left\lVert\mu_\varepsilon(y,X^{\varepsilon,x}(s))\right\rVert_{L^{pr}(\P;\R^d)} \bigr|_{y=X^{\varepsilon,x}(s)}
\right\rVert_{L^{pr}(\P;\R^d)}
ds +\left\lVert W(t)\right\rVert_{L^{pr}(\P;\R^d)}\nonumber\\
&
\leq \lVert x\rVert+\int_{0}^{t}
\left\lVert\mu_\varepsilon(X^{\varepsilon,x}(s),X^{\varepsilon,x}(s))\right\rVert_{L^{pr}(\P;\R^d)}ds+\xeqref{a05}
\sqrt{t(d+2pr)}\nonumber\\
&\leq \lVert x\rVert+\int_{0}^{t}\left(\xeqref{a02}
\left\lVert\mu_\varepsilon(0,0)\right\rVert+c
\left\lVert X^{\varepsilon,x}(s)\right\rVert_{L^{pr}(\P;\R^d)}\right)ds+\sqrt{t(d+2pr)}\nonumber\\
&\leq \lVert x\rVert+T\left\lVert\mu_\varepsilon(0,0)\right\rVert+\sqrt{T(d+2pr)}+
c
\int_{0}^{t}
\left\lVert X^{\varepsilon,x}(s)\right\rVert_{L^{pr}(\P;\R^d)}ds.
\end{align}
This, the fact that
$\forall\,\varepsilon\in [0,1), x\in\R^d\colon \int_{0}^{T} \left\lVert X^{\varepsilon,x}(t)\right\rVert_{L^{pr}(\P;\R^d)}\,dt <\infty$, and Gr\"onwall's inequality prove for all
$\varepsilon\in [0,1)$, 
$t\in [0,T]$,
$x\in\R^d$ that
\begin{align}
\left\lVert 
X^{\varepsilon,x}(t)\right\rVert_{L^{pr}(\P;\R^d)}\leq
\left(
\lVert x\rVert+T\left\lVert\mu_\varepsilon(0,0)\right\rVert+\sqrt{T(d+2pr)}\right)e^{ct}.\label{a07}
\end{align}
This shows \eqref{a07b}.

Next, \eqref{a03}, \eqref{a07}, and the triangle inequality prove for all 
$\varepsilon\in [0,1)$, 
$s\in [0,T]$,
$x\in\R^d$ that
\begin{align}  
&
\left\lVert\mu_\varepsilon(X^{0,x}(s),X^{0,x}(s))-\mu_0(X^{0,x}(s),X^{0,x}(s))\right\rVert_{L^p(\P;\R^d)}\nonumber
\\
&=
\Bigl\lVert
\left\lVert
\mu_\varepsilon(X^{0,x}(s),X^{0,x}(s))-\mu_0(X^{0,x}(s),X^{0,x}(s))
\right\rVert
\Bigr\rVert_{L^p(\P;\R)}\nonumber
\\
&\leq \xeqref{a03}
\left\lVert b\varepsilon+
0.5b\varepsilon\left\lVert X^{0,x}(s)\right\rVert^r+
0.5b\varepsilon\left\lVert X^{0,x}(s)\right\rVert^r
\right\rVert_{L^p(\P;\R)}\nonumber
\\
&\leq b\varepsilon+b\varepsilon
\left\lVert \left\lVert X^{0,x}(s)\right\rVert^r
\right\rVert_{L^p(\P;\R)}\nonumber\\
&
= b\varepsilon+b\varepsilon
\left(\E\!\left[ \left\lVert X^{0,x}(s)\right\rVert^{pr}\right]\right)^\frac{1}{p}\nonumber\\&=
b\varepsilon+
b\varepsilon\left\lVert X^{0,x}(s)\right\rVert_{L^{pr}(\P;\R^d)}^r\nonumber\\&\leq
b\varepsilon+ b\varepsilon\xeqref{a07}
\left(
\lVert x\rVert+T\left\lVert\mu_0(0,0)\right\rVert+\sqrt{T(d+2pr)}\right)^re^{rcT}\nonumber\\
&
\leq 
b\varepsilon
\left(1+
\lVert x\rVert+T\left\lVert\mu_0(0,0)\right\rVert+\sqrt{T(d+2pr)}\right)^re^{rcT}
.
\label{a08}\end{align}
Hence, 
\eqref{a04}, 
the triangle inequality, Jensen's inequality, the disintegration theorem, and \eqref{a02} show for all
$\varepsilon\in [0,1)$, 
$t\in [0,T]$,
$x\in\R^d$ that
\begin{align}  
&\left\lVert
X^{\varepsilon,x}(t)-X^{0,x}(t)\right\rVert_{L^p(\P;\R^d)}\nonumber\\
&
\leq \int_{0}^{t}\left\lVert
\E\!\left[\mu_\varepsilon(y,X^{\varepsilon,x}(s))\right]\bigr|_{y=X^{\varepsilon,x}(s)}
-
\E\!\left[\mu_0(z,X^{0,x}(s))\right]\bigr|_{z=X^{0,x}(s)}
\right\rVert_{L^p(\P;\R^d)} ds\nonumber\\
& \leq \int_{0}^{t}\left\lVert
\E\!\left[\mu_\varepsilon(y,X^{\varepsilon,x}(s))-\mu_0(z,X^{0,x}(s))\right]\bigr|_{y=X^{\varepsilon,x}(s) ,z=X^{0,x}(s)}
\right\rVert_{L^p(\P;\R^d)} ds\nonumber\\
& = \int_{0}^{t}\left\lVert
\left\lVert
\E\!\left[\mu_\varepsilon(y,X^{\varepsilon,x}(s))-\mu_0(z,X^{0,x}(s))\right]\right\rVert\bigr|_{y=X^{\varepsilon,x}(s) ,z=X^{0,x}(s)}
\right\rVert_{L^p(\P;\R)} ds\nonumber\\
& \leq  \int_{0}^{t}\left\lVert
\left\lVert\mu_\varepsilon(y,X^{\varepsilon,x}(s))-\mu_0(z,X^{0,x}(s))\right\rVert_{L^p(\P;\R^d)}\bigr|_{y=X^{\varepsilon,x}(s) ,z=X^{0,x}(s)}
\right\rVert_{L^p(\P;\R)} ds\nonumber\\
&=
\int_{0}^{t}
\left\lVert\mu_\varepsilon(X^{\varepsilon,x}(s),X^{\varepsilon,x}(s))-\mu_0(X^{0,x}(s),X^{0,x}(s))\right\rVert_{L^p(\P;\R^d)}ds\nonumber\\
&\leq 
\int_{0}^{t}\Bigl[
\left\lVert\mu_\varepsilon(X^{\varepsilon,x}(s),X^{\varepsilon,x}(s))-\mu_\varepsilon(X^{0,x}(s),X^{0,x}(s))\right\rVert_{L^p(\P;\R^d)}\nonumber\\
&\qquad \qquad\qquad 
+
\left\lVert\mu_\varepsilon(X^{0,x}(s),X^{0,x}(s))-\mu_0(X^{0,x}(s),X^{0,x}(s))\right\rVert_{L^p(\P;\R^d)}
\Bigr]
ds\nonumber\\
&\leq c\int_{0}^{t}
\left\lVert X^{\varepsilon,x}(s)-X^{0,x}(s)
\right\rVert_{L^p(\P;\R^d)}
ds +T\xeqref{a08}b\varepsilon \left(1+
\lVert x\rVert+T\left\lVert\mu_0(0,0)\right\rVert+\sqrt{T(d+2pr)}\right)^r e^{rcT}.
\end{align}
This, 
the fact that
$\forall\,\varepsilon\in [0,1), x\in\R^d\colon \int_{0}^{T} \left\lVert X^{\varepsilon,x}(t)\right\rVert_{L^{pq}(\P;\R^d)}dt <\infty$, and
Gr\"onwall's inequality show for all
$\varepsilon\in [0,1)$, 
$t\in [0,T]$,
$x\in\R^d$ that
\begin{align} \begin{split} 
\left\lVert
X^{\varepsilon,x}(t)-X^{0,x}(t)\right\rVert_{L^p(\P;\R^d)}&\leq 
Tb\varepsilon \left(1+
\lVert x\rVert+T\left\lVert\mu_0(0,0)\right\rVert+\sqrt{T(d+2pr)}\right)^re^{rcT}e^{ct}\\
&\leq
Tb\varepsilon \left(1+
\lVert x\rVert+T\left\lVert\mu_0(0,0)\right\rVert+\sqrt{T(d+2pr)}\right)^re^{(r+1)cT}.
 \end{split}\label{a12}
\end{align}
This shows \eqref{a11}. 

Next, \eqref{a04b}, the triangle inequality, and  \eqref{a07} show for all $\varepsilon\in (0,1)$, $t\in [0,T]$, $x\in \R^d$ that
\begin{align} 
\left\lVert
f_\varepsilon(X^{\varepsilon,x}(t))
-
f_0(X^{\varepsilon,x}(t))\right\rVert_{L^p(\P;\R)}
&\leq \varepsilon\xeqref{a04b}\left(
\left\lVert
1+
\left\lVert
X^{\varepsilon,x}(t)
\right\rVert^r
\right\rVert_{L^p(\P;\R)}
\right)\nonumber\\&\leq 
\varepsilon\left(
1+\left\lVert
\left\lVert
X^{\varepsilon,x}(t)
\right\rVert^r
\right\rVert_{L^p(\P;\R)}
\right)\nonumber\\
&=
\varepsilon\left(
1+\left\lVert
X^{\varepsilon,x}(t)
\right\rVert_{L^{pr}(\P;\R^d)}^r
\right)\nonumber
\\&\leq 
\varepsilon\left(
1+\left(
\xeqref{a07}
\left(
\lVert x\rVert+T\left\lVert\mu_\varepsilon(0,0)\right\rVert+\sqrt{T(d+2pr)}\right)e^{cT}
\right)^r
\right)\nonumber
\\
&=\varepsilon\left(
1+
\left(
\lVert x\rVert+T\left\lVert\mu_\varepsilon(0,0)\right\rVert+\sqrt{T(d+2pr)}\right)^re^{rcT}
\right)\nonumber\\
&\leq 
\varepsilon
\left(1+
\lVert x\rVert+T\left\lVert\mu_\varepsilon(0,0)\right\rVert+\sqrt{T(d+2pr)}\right)^re^{rcT}.
\label{a14}
\end{align}
In addition, \eqref{a04b} and \eqref{a12}
prove for all $\varepsilon\in (0,1)$, $t\in [0,T]$, $x\in \R^d$ that
\begin{align} \begin{split} 
\left\lVert
f_0(X^{\varepsilon,x}(t))-
f_0(X^{0,x}(t))\right\rVert_{L^p(\P;\R)}&\leq \xeqref{a04b}c
\left\lVert
X^{\varepsilon,x}(t)-X^{0,x}(t)
\right\rVert_{L^p(\P;\R^d)}\\
&
\leq c \xeqref{a12}
Tb\varepsilon \left(1+
\lVert x\rVert+T\left\lVert\mu_0(0,0)\right\rVert+\sqrt{T(d+2pr)}\right)^re^{(r+1)cT}.
\end{split}\label{a15}
\end{align}
This, the triangle inequality,  \eqref{a14}, and the fact that $1+cT\leq e^{cT}$
imply for all $\varepsilon\in (0,1)$, $t\in [0,T]$, $x\in \R^d$ that
\begin{align} \begin{split} 
&
\left\lVert f_\varepsilon(X^{\varepsilon,x}(t))
-
f_0(X^{0,x}(t))\right\rVert_{L^p(\P;\R)}
\\&\leq 
\left\lVert
f_\varepsilon(X^{\varepsilon,x}(t))
-
f_0(X^{\varepsilon,x}(t))\right\rVert_{L^p(\P;\R)}
+
\left\lVert
f_0(X^{\varepsilon,x}(t))-
f_0(X^{0,x}(t))\right\rVert_{L^p(\P;\R)}
\\
&\leq \xeqref{a14}
\varepsilon
\left(1+
\lVert x\rVert+T\left\lVert\mu_\varepsilon(0,0)\right\rVert+\sqrt{T(d+2pr)}\right)^re^{rcT}\\
&\quad 
+\xeqref{a15}
c 
Tb\varepsilon \left(1+
\lVert x\rVert+T\left\lVert\mu_0(0,0)\right\rVert+\sqrt{T(d+2pr)}\right)^re^{(r+1)cT}\\
&\leq 
b\varepsilon \left(1+
\lVert x\rVert+T
\max\{
\left\lVert\mu_\varepsilon(0,0)\right\rVert,
\left\lVert\mu_0(0,0)\right\rVert
\}
+\sqrt{T(d+2pr)}\right)^re^{(r+2)cT}.
\end{split}
\end{align}
This shows \eqref{c01}.
The proof of \cref{a01} is thus completed.
\end{proof}

\section{Deep Neural Networks}\label{s04}
\subsection{Basic DNN calculus}
For the proof of our main result in this section, \cref{k12,s14}, we need several auxiliary results, Lemmas \ref{k43}--\ref{b04}, which are basic facts on DNNs. The proof of \cref{k43}--\ref{b04} can be found in 
\cite{CHW2022,HJKN2020a,NNW2023} and therefore omitted.
\begin{setting}\label{m07b}
Assume \cref{m07}.
Let $\odot \colon \mathbf{D}\times \mathbf{D} \to\mathbf{D} $ satisfy 
for all $H_1,H_2\in \N$, $ \alpha=(\alpha_0,\alpha_1,\ldots,\alpha_{H_1},\alpha_{H_1+1})\in\N^{H_1+2}$, $\beta=(\beta_0,\beta_1,\ldots,\beta_{H_2},\beta_{H_2+1})\in\N^{H_2+2}$
that
$
\alpha\odot \beta= (\beta_{0},\beta_{1},\ldots,\beta_{H_2},\beta_{H_2+1}+\alpha_{0},\alpha_{1},\alpha_{2},\ldots,\alpha_{H_1+1})\in \N^{H_1+H_2+3}
$.
Let $\boxplus \colon \mathbf{D}\times \mathbf{D} \to\mathbf{D}  $ satisfy
for all $H\in \N$, 
$\alpha= (\alpha_0,\alpha_1,\ldots,\alpha_{H},\alpha_{H+1})\in \N^{H+2}$,
$\beta= (\beta_0,\beta_1,\beta_2,\ldots,\beta_{H},\beta_{H+1})\in \N^{H+2}$
that
$
\alpha \boxplus \beta =(\alpha_0,\alpha_1+\beta_1,\ldots,\alpha_{H}+\beta_{H},\beta_{H+1})\in \N^{H+2}.
$
Let $\boxdot\colon \mathbf{D}\times\mathbf{D}\to \mathbf{D}$ satisfy for all $H\in \N$, 
$\alpha= (\alpha_0,\alpha_1,\ldots,\alpha_{H},\alpha_{H+1})\in \N^{H+2}$,
$\beta= (\beta_0,\beta_1,\beta_2,\ldots,\beta_{H},\beta_{H+1})\in \N^{H+2}$
that $\alpha \boxdot \beta =(\alpha_{0},\alpha_1+\beta_1,\ldots,\alpha_H+\beta_H, \alpha_{H+1}+\beta_{H+1} )\in \N^{H+2}$.
Let $\mathfrak{n}_n^d
\in \mathbf{D} $, $n\in [3,\infty)\cap\Z$, $d\in \N$,  satisfy for all $n\in [3,\infty)\cap\N$, $d\in \N$ that
\begin{align}
\label{k07}
\mathfrak{n}_n^d= (d,\underbrace{2d,\ldots,2d}_{(n-2)\text{ times}},d)\in \N^n
.
\end{align} Let $\mathfrak{n}_n\in \mathbf{D}$, $n\in [3,\infty)$, satisfy for all
$n\in [3,\infty)$ that $\mathfrak{n}_n=\mathfrak{n}^1_n$.

\end{setting}

\begin{lemma}[$\odot$ is associative--{\cite[Lemma 3.3]{HJKN2020a}}]\label{k43}Assume \cref{m07b} and let $\alpha,\beta,\gamma\in \bfD$. Then we have that
$(\alpha\odot\beta)\odot \gamma
= \alpha\odot(\beta\odot \gamma)$.
\end{lemma}

\begin{lemma}[$\boxplus$ and associativity--{\cite[Lemma 3.4]{HJKN2020a}}]Assume \cref{m07b},
let $H,k,l \in \N$, and let $\alpha,\beta,\gamma\in \left( \{k\}\times \N^{H} \times \{l\}\right)$.
Then
\begin{enumerate}[(i)]
\item we have that $\alpha\boxplus\beta\in \left(\{k\}\times \N^{H} \times \{l\}\right)$,
\item we have that $\beta\boxplus \gamma\in \left(\{k\}\times \N^{H} \times \{l\}\right)$, and 
\item we have that $(\alpha\boxplus\beta)\boxplus \gamma
= \alpha\boxplus(\beta\boxplus \gamma)$.
\end{enumerate}
\end{lemma}
 \begin{lemma}[Triangle inequality--{\cite[Lemma 3.5]{HJKN2020a}}]\label{b15}
Assume \cref{m07b},
let $H,k,l \in \N$, and let $\alpha,\beta\in \{k\}\times \N^{H} \times \{l\}$.
Then we have that
$\supnorm{\alpha\boxplus\beta}\leq\supnorm{\alpha}+
\supnorm{\beta} $.
\end{lemma}
\begin{lemma}[DNNs for affine transformations--{\cite[Lemma 3.7]{HJKN2020a}}]\label{p01}
Assume \cref{m07} and let $d,m\in \N$, 
$\lambda\in \R$,
$b\in\R^d$, $a\in\R^m$, $\Psi\in\mathbf{N}$ satisfy that $\mathcal{R}(\Psi)\in C(\R^d,\R^m)$. Then we have that
$
\lambda\left((\mathcal{R}(\Psi))(\cdot +b)+a\right)\in \mathcal{R}(\{\Phi\in\mathbf{N}\colon \mathcal{D}(\Phi)=\mathcal{D}(\Psi)\}).
$
\end{lemma}



\begin{lemma}[Composition of functions generated by DNNs--{\cite[Lemma 3.8]{HJKN2020a}}]\label{m11b}
Let \cref{m07b} be given and let $d_1,d_2,d_3\in\N$, $f\in C(\R^{d_2},\R^{d_3})$, $g\in C(  \R^{d_1}, \R^{d_2}) $, 
$\alpha,\beta\in \mathbf{D}$ satisfy that
$f\in \mathcal{R}(\{\Phi\in \mathbf{N}\colon \mathcal{D}(\Phi)=\alpha\})$
and
$g\in \mathcal{R}(\{\Phi\in \mathbf{N}\colon \mathcal{D}(\Phi)=\beta\})$.
Then we have
that $(f\circ g)\in \mathcal{R}(\{\Phi\in \mathbf{N}\colon \mathcal{D}(\Phi)=\alpha\odot\beta\})$.
\end{lemma}

\begin{lemma}[Sum of DNNs of the same length--{\cite[Lemma 3.9]{HJKN2020a}}]
\label{b01b}
Assume \cref{m07b} and let $M,H,p,q\in \N$,  $h_1,h_2,\ldots,h_M\in\R$,
 $k_i\in \mathbf{D} $,
$f_i\in C(\R^{p},\R^{q})$,
$i\in [1,M]\cap\N$, satisfy 
for all $i\in [1,M]\cap\N$
that $ \dim(k_i)=H+2$ and
$f_i\in 
\mathcal{R}(\{\Phi\in\mathbf{N}\colon \mathcal{D}(\Phi)=k_i\}).
$
Then
we have that 
$
\sum_{i=1}^{M}h_if_i
\in\mathcal{R}\left(\left\{ \Phi\in\mathbf{N}\colon
\mathcal{D}(\Phi)=\boxplus_{i=1}^Mk_i\right\}\right).
$
\end{lemma}

\begin{lemma}[DNNs with $H$ hidden layers for $\mathrm{Id}_{\R^d}$--{\cite[Lemma 3.6]{CHW2022}}]\label{b03}
Assume \cref{m07b} and let $d,H\in \N$.
Then we have that
$\mathrm{Id}_{\R^d}\in \mathcal{R}(\{\Phi\in\mathbf{N}\colon\mathcal{D}(\Phi)=\mathfrak{n}^d_{H+2} \}) $.
\end{lemma}

\begin{lemma}[{\cite[Lemma 3.7]{CHW2022}}]\label{b02}
Assume \cref{m07}, let $H,p,q\in \N$, and let $g\in C(\R^p,\R^q)$ satisfy that
$g\in \calR(\{\Phi\in \bfN\colon \dim(\calD(\Phi))=H+2\})$. Then for all
$n\in \N_0$ we have that
$g\in \calR(\{\Phi\in \bfN\colon \dim(\calD(\Phi))=H+2+n\})$.
\end{lemma}
\begin{lemma}[{\cite[Lemma 4.10]{NNW2023}}]\label{b04}
Assume \cref{m07b}. Then for all $n\in \N $, $d_0,d_1,\ldots, d_n\in \N$, 
$f_1\in C(\R^{d_1},\R^{d_{0}}), f_2\in C(\R^{d_2},\R^{d_{1}}), \ldots ,f_n\in C(\R^{d_n},\R^{d_{n-1}})$, 
 $\phi_1,\phi_2,\ldots,\phi_n\in \bfN$ 
with 
 $\forall i\in [1,n]\cap\Z\colon f_i=\calR(\phi_i)$ we have that
\begin{align}
\supnorm{\operatorname*{\odot}_{i=1}^n\mathcal{D}(\phi_i)}\leq \max\left\{
\supnorm{\calD(\phi_1)},
\supnorm{\calD(\phi_2)},\ldots,\supnorm{\calD(\phi_n)}, 2d_1,2d_2,\ldots,2d_{n-1}
\right\}.
\end{align}
\end{lemma}

The following result, \cref{p11}, is a modification of \cite[Lemma 3.9]{HJKN2020a}.
\begin{lemma}[DNN representing the merging]\label{p11}
Assume \cref{m07b}, 
let $M,H,p\in \N$, and for every
$i\in [1,M]\cap\Z$ let $q_i\in \N$, $f_i\in C(\R^p,\R^{q_i})$, $k_i\in \mathbf{D}$ satisfy that
$\dim (k_i)=H+2$ and $ f_i\in\mathcal{R} (\{\Phi\in\mathbf{N}\colon \mathcal{D}(\Phi)=k_i\})$. Then $(f_1,\ldots,f_M)^{ \top}\in\mathcal{R} (\{\Phi\in\mathbf{N}\colon \mathcal{D}(\Phi)=\boxdot_{i=1}^M k_i\}) $.
\end{lemma}
\begin{proof}[Proof of \cref{p11}]
Let $\phi_i\in \mathbf{N}$, $i\in [1,M]\cap \Z$, and $k_{i,n}\in \N$,
$i\in [1,M]\cap\Z$, $n\in [0,H+1]\cap\Z$, 
satisfy for all $i\in [1,M]\cap\Z$ that
$\mathcal{D}(\phi_i)=k_i=(k_{i,0},\ldots,k_{i,H+1})$
and $\mathcal{R}(\phi_i)=f_i$, for every $i\in [1,M]\cap\Z$ let 
$((W_{i,1},B_{i,1}),\ldots,(W_{i,H+1},B_{i,H+1}) )\in \prod_{n=1}^{H+1}
\R^{k_{i,n}\times k_{i,n-1} } \times \R^{k_{i,n}}
$ satisfy that
$\phi_i = ((W_{i,1},B_{i,1}),\ldots, (W_{i,H+1},B_{i,H+1}) )$,
let $k_n^\boxdot\in \N$, $n\in [1,H+1]\cap\N$, $k^\boxdot\in\N^{H+2} $
satisfy for all $n\in [1,H+1]\cap\N$ that
\begin{align}
k^\boxdot_n=\sum_{i=1}^{M}k_{i,n} \quad\text{and}\quad 
k^\boxdot=(p,k_1^\boxdot,\ldots, k_{H+1}^\boxdot),
\end{align}
let $W_1\in \R^{k_1^\boxdot\times p}$, $B_1\in \R^{k_1^\boxdot}$
satisfy that

\begin{align}
W_1= \begin{pmatrix}
W_{1,1}\\
W_{2,1}\\
\vdots\\
W_{M,1}
\end{pmatrix}, B_1=\begin{pmatrix}
B_{1,1}\\
B_{2,1}\\
\vdots\\
B_{M,1}
\end{pmatrix},
\end{align}
let $W_n\in \R^{k_n^{\boxdot}\times k_{n-1}^{\boxdot}}$,
$B_n\in \R^{k_n^{\boxdot}}$, $n\in [2,H+1]\cap\Z$, satisfy for all
$n\in [2,H+1]\cap\Z$ that

\begin{align}\begin{split}
W_n= \begin{pmatrix}
W_{1,n} &	0		&	0		&	0	\\
0 		&	W_{2,n}	&	0		&	0	\\
0		&	0		&	\ddots	&	0	\\
0		&	0		&	0		&W_{M,n}
\end{pmatrix}
\quad\text{and}\quad
B_n=
\begin{pmatrix}
B_{1,n}\\
B_{2,n}\\
\vdots\\
B_{M,n}
\end{pmatrix},\end{split}\label{b01d}
\end{align}
let $x_0\in\R^p,\, x_1\in \R^{k_1^{\boxdot}},
x_2\in \R^{k_2^{\boxdot}}
\ldots,x_H\in \R^{k_H^{\boxdot}}$, 
let 
$x_{1,0},x_{2,0},\ldots,x_{M,0}\in \R^{p}$,
 $x_{i,n}\in \R^{k_{i,n}}$, $i\in [1,M]\cap\N$, $n\in [1,H]\cap\N$, satisfy 
for all $i\in [1,M]\cap\N$, $n\in [1,H]\cap\N$
that 
\begin{align}\begin{split}
&x_0=x_{1,0}=x_{2,0}=\ldots=x_{M,0},\\
&x_{i,n}=\mathbf{A}_{k_{i,n}}(W_{i,n}x_{i,n-1}+B_{i,n}),\\ 
&x_n= \mathbf{A}_{k^{\boxdot}_{n}}(W_{n}x_{n-1}+B_{n}),
\end{split}
\end{align}
and let $\psi\in \mathbf{N}$ satisfy that
\begin{align}
\psi= \left((W_1,B_1),(W_2,B_2),\ldots,(W_H,B_H),(W_{H+1},B_{H+1})\right).
\end{align}
First, 
the definitions of $\mathcal{D}$ and $\mathcal{R}$, 
the choice of $\phi_i$, and the fact that $\forall \, i\in [1,M]\cap\N\colon f_i\in C(\R^p,\R^{q_i})$ show for all $i\in [1,M]\cap\N$ that 
$k_{i,H+1}=q_i$ and
$
k_i=
(p,k_{i,1},k_{i,2},\ldots,k_{i,H+1}).
$
The definition of $\mathcal{R}$,
$\boxdot$,
and $k^\boxdot$ then show that
\begin{align}
\mathcal{D}(\psi)= (p,k_1^{\boxdot},\ldots,k_{H+1}^{\boxdot})=\boxdot_{i=1}^Mk_i.\label{b23}
\end{align}
Next, we prove by induction on $n\in [1,H]\cap\N$ that $ x_n=(x_{1,n},x_{2,n},\ldots,x_{M,n})$.
First, the definition of $(W_1,B_1)$ shows  that
\begin{align}
W_1x_0+B_1= 
\begin{pmatrix}
W_{1,1}\\
W_{2,1}\\
\vdots\\
W_{M,1}
\end{pmatrix}x_0+
\begin{pmatrix}
B_{1,1}\\
B_{2,1}\\
\vdots\\
B_{M,1}
\end{pmatrix}
=
\begin{pmatrix}
W_{1,1}x_0+B_{1,1}\\
W_{2,1}x_0+B_{2,1}\\
\vdots\\
W_{M,1}x_0+B_{M,1}
\end{pmatrix}.\label{b20}
\end{align}
This implies that
\begin{align}
x_1= \mathbf{A}_{k_1^{\boxdot}}(W_1x_0+B_1)=\begin{pmatrix}
x_{1,1}\\x_{2,1}\\\vdots\\x_{M,1}\end{pmatrix}.
\end{align}
This proves the base case. Next, for the induction step let $n\in [2,H]\cap\N$ and assume that $x_{n-1}=(x_{1,n-1},x_{2,n-1},\ldots,x_{M,n-1})$.
Then \eqref{b01d} and the induction hypothesis ensure that
\begin{align}\begin{split}
&W_nx_{n-1}+B_n
\\
&= W_{n}\begin{pmatrix}
x_{1,n-1}\\
x_{2,n-1}\\
\vdots\\
x_{M,n-1}
\end{pmatrix}+B_{n}
=\begin{pmatrix}
W_{1,n} &	0		&	0		&	0	\\
0 		&	W_{2,n}	&	0		&	0	\\
0		&	0		&	\ddots	&	0	\\
0		&	0		&	0		&W_{M,n}
\end{pmatrix}
\begin{pmatrix}
x_{1,n-1}\\
x_{2,n-1}\\
\vdots\\
x_{M,n-1}
\end{pmatrix}+
\begin{pmatrix}
B_{1,n}\\
B_{2,n}\\
\vdots\\
B_{M,n}
\end{pmatrix}
\\
&=
\begin{pmatrix}
W_{1,n}x_{1,n-1}+
B_{1,n}\\
W_{2,n}x_{2,n-1}+B_{2,n}\\
\vdots\\
W_{M,n}x_{M,n-1}+ B_{M,n}
\end{pmatrix}.\label{b21}\end{split}
\end{align}
This yields that
\begin{align}
x_{n}= \mathbf{A}_{k_n^{\boxdot}}(W_nx_{n-1}+B_n)=\begin{pmatrix}
x_{1,n}\\x_{2,n}\\\vdots\\x_{M,n}
\end{pmatrix}.
\end{align}
This proves the induction step. Induction now proves for all $n\in [1,H]\cap\N$ that
$x_n=(x_{1,n},x_{2,n},\ldots,x_{M,n})$.
This and the definition of $\mathcal{R}$ imply that
\begin{align}
&(\mathcal{R}(\psi))(x_0)=W_{H+1}x_H+B_{H+1}\nonumber\\
&=W_{H+1}\begin{pmatrix}
x_{1,H}\\
x_{2,H}\\
\vdots\\
x_{M,H}
\end{pmatrix}+B_{H+1}\nonumber
\\&=\begin{pmatrix}
W_{1,H+1} &	0		&	0		&	0	\\
0 		&	W_{2,H+1}	&	0		&	0	\\
0		&	0		&	\ddots	&	0	\\
0		&	0		&	0		&W_{M,H+1}
\end{pmatrix}
\begin{pmatrix}
x_{1,H}\\
x_{2,H}\\
\vdots\\
x_{M,H}
\end{pmatrix}+\begin{pmatrix}
B_{1,H+1}\\
B_{2,H+1}\\
\vdots\\
B_{M,H+1}
\end{pmatrix}
\nonumber\\
&=\begin{pmatrix}
W_{1,H+1}x_{1,H}+B_{1,H+1}\\
W_{2,H+1}x_{2,H}+B_{2,H+1}\\
\vdots\\
W_{M,H+1}x_{M,H}+B_{M,H+1}
\end{pmatrix}
=\begin{pmatrix}
\mathcal{R}(\phi_1)(x_{1,0})\\
\mathcal{R}(\phi_2)(x_{2,0})\\
\vdots\\
\mathcal{R}(\phi_M)(x_{M,0})\\
\end{pmatrix}=
\begin{pmatrix}
f_1(x_{0})\\
f_2(x_{0})\\
\vdots\\
f_M(x_{0})\\
\end{pmatrix}
\end{align}
The proof of \cref{p11} is thus completed.
\end{proof}
\subsection{DNN representation of  MLP approximations}
In \cref{s12} below we introduce MLP approximations for the solutions of McKean--Vlasov SDEs in the case when the drift coefficient $\mu$ is a DNN function, see \eqref{a05b}. In \cref{k12} we show that those MLP approximations can also be represented by DNNs. 
\begin{setting}\label{s12}
Assume \cref{m07}. Let 
$T\in (0,\infty)$, 
$m,d\in\N$,
$\mu \in C(\R^d\times\R^d,\R^d)$, 
$\Phi_\mu\in \bfN$,
$\Theta= \cup_{n\in\N}(\N_0)^n$
satisfy that
$\mu=\calR(\Phi_\mu)$.
Let $(\Omega,\mathcal{F},\P)$ be a  probability space, 
Let $\mathfrak{t}^\theta\colon\Omega\to[0,1]$, $\theta\in\Theta$, be i.i.d.\ random variables. Assume for all $t\in [0,1] $ that
$\P(\mathfrak{t}^0\leq t)=t$. Let $W^\theta\colon [0,T]\times\Omega\to\R^d$, $\theta\in\Theta$, be i.i.d.\ standard 
Brownian motions with continuous sample paths.
Assume that 
$(\mathfrak{t}^\theta)_{ \theta\in\Theta}$ and
$(W^\theta)_{ \theta\in\Theta}$ 
are independent. For every $x\in \R^d$
let $X^{\theta,x}_{n}\colon [0,T]\times\Omega\to\R^d$, $\theta\in\Theta$,  $n\in \N_0$, satisfy  for all
$\theta\in\Theta$, $n\in\N$, $t\in [0,T] $  that
$
X_{0}^{\theta,x}(t)=0
$
and \begin{align}\label{a05b} 
\begin{split}
&
X_{n}^{\theta,x}(t)= x+ W^\theta\left(\sup\!\left(\left\{\frac{kT}{m^n}\colon k\in\N_0\right\}\cap[0,t]\right)\right)+t\mu(0,0)\\& + \sum_{\ell=1}^{n-1}
\sum_{k=1}^{m^{n-\ell}}\Biggl[\frac{t\left[
\mu\bigl(X^{\theta,x}_{\ell}(\mathfrak{t}^{(\theta,n,k,\ell)}t),
X^{(\theta,n,k,\ell),x}_{\ell}(\mathfrak{t}^{(\theta,n,k,\ell)}t)
\bigr)-
\mu\bigl(X^{\theta,x}_{\ell-1}(\mathfrak{t}^{(\theta,n,k,\ell)}t),
X^{(\theta,n,k,\ell),x}_{\ell-1}(\mathfrak{t}^{(\theta,n,k,\ell)}t)
\bigr)\right]
}{m^{n-\ell}}
\Biggr]
.
\end{split}\end{align}
Let $\omega\in \Omega$.
\end{setting}
For the proof of \cref{k12} below we use the techniques learned from the proof of \cite[Lemma 3.10]{HJKN2020a}.
\begin{lemma}\label{k12}
Assume \cref{s12} and
 let $c\in \R$ satisfy that
\begin{align}
2c\geq \max\{4d,\supnorm{\Phi_\mu}\}.\label{k27}
\end{align}
Then for all $n\in \N_0$ there exists $(\Phi_{n,t}^\theta)_{t\in [0,T],\theta\in\Theta}\subseteq \bfN$ such that the following items hold.
\begin{enumerate}[(i)]
\item For all $t_1,t_2\in[0,T]$, $\theta_1,\theta_2\in \Theta$ we have that
$\calD(\Phi_{n,t_1}^{\theta_1})=\calD(\Phi_{n,t_2}^{\theta_2})$.
\item For all $t\in[0,T]$, $\theta\in \Theta$ we have that
$\dim(\calD( \Phi_{n,t}^\theta ))= n(\dim (\calD (\Phi_\mu))-1)+3$.
\item For all $t\in[0,T]$, $\theta\in \Theta$ we have that $\supnorm{\calD( \Phi_{n,t}^\theta )} \leq c(5m)^n$. 
\item For all $t\in[0,T]$, $\theta\in \Theta$, $x\in\R^d$ we have that
$X^{\theta,x}_{n}(t,\omega)=(\calR( \Phi_{n,t}^\theta ))(x)$.
\end{enumerate}
\end{lemma}
\begin{proof}
[Proof of \cref{k12}]
We prove this result by induction 
on $n\in \N_0$. 
Throughout this proof let the notation in \cref{m07b} be given.
First, the fact that the zero function can be represented by a DNN of arbitrary length proves the base case $n=0$. For the induction step $\N_0\ni n\mapsto n+1\in \N$ let $n\in \N_0$, 
$(\Phi_{\ell,t}^\theta)_{\ell\in [0,n]\cap\Z, t\in [0,T],\theta\in\Theta}\subseteq \bfN$ satisfy 
 for all $t,t_1,t_2\in[0,T]$, $x\in \R^d$, $\theta,\theta_1,\theta_2\in \Theta$, $\ell\in [0,n]\cap\Z$  that
\begin{align} \begin{split}& \calD(\Phi_{\ell,t_1}^{\theta_1})=\calD(\Phi_{\ell,t_2}^{\theta_2}),
\quad\dim(\calD( \Phi_{\ell,t}^\theta ))= \ell(\dim (\calD (\Phi_\mu))-1)+3,
\\&
\supnorm{\calD( \Phi_{\ell,t}^\theta )} \leq c(5m)^\ell,\quad 
X^{\theta,x}_{\ell}(t,\omega)=(\calR( \Phi_{\ell,t}^\theta ))(x).
\end{split}
\label{k29}
\end{align}
First,
\cref{b03} and
\cref{p01} show for all $t\in [0,T]$, $\theta\in\Theta$ that
\begin{align} \begin{split} 
&
\left[
\R^d\ni x\mapsto
x+ W^\theta\left(\sup\!\left(\left\{\frac{kT}{m^{n+1}}\colon k\in\N_0\right\}\cap[0,t]\right),\omega\right)+t\mu(0,0)\in \R^d
\right]\\
&\in \calR\left(\Bigl\{
\Phi\in\bfN\colon 
\calD(\Phi)=\mathfrak{n}^d_{
(n+1)(\dim (\calD (\Phi_\mu))-1)+3}
\Bigr\}\right).
\end{split}\label{k30}\end{align}
Next, observe that \eqref{k07} implies that
\begin{align}
\dim (\mathfrak{n}^d_{
(n+1)(\dim (\calD (\Phi_\mu))-1)+3})=
(n+1)(\dim (\calD (\Phi_\mu))-1)+3.
\end{align}
Moreover, the fact that $\mu=\calR(\Phi_\mu)$, \eqref{k29}, \cref{m11b},  \cref{p11}, and the definition of $\odot$ and $\boxdot$ 
show for all $t\in [0,T]$, $\theta\in\Theta$ that
\begin{align} \begin{split} 
&
\left[
\R^d\ni x\mapsto
\mu\bigl(X^{\theta,x}_{n}(\mathfrak{t}^{(\theta,n+1,k,n)}(\omega)t,\omega),
X^{(\theta,n+1,k,n),x}_{n}(\mathfrak{t}^{(\theta,n+1,k,n)}(\omega)t,\omega)
\bigr)\in \R^d
\right]\\&\in \calR\left(\Bigl\{
\Phi\in\bfN\colon\calD(\Phi)=
 \calD(\Phi_\mu)\odot\left(\calD(\Phi_{n,0}^0)\boxdot\calD(\Phi_{n,0}^0)\right)
\Bigr\}\right)
\end{split}
\end{align} and
\begin{align} 
&
\dim\!\left(
\calD(
\Phi_\mu)\odot\left(\calD(\Phi_{n,0}^0)\boxdot\calD(\Phi_{n,0}^0)\right)\right)\nonumber\\&=
\dim(\calD(\Phi_\mu))+\dim\!\left(\calD(\Phi_{n,0}^0)\boxdot\calD(\Phi_{n,0}^0)  \right)-1\nonumber
\\
&=\dim(\calD(\Phi_\mu))+
\dim\!\left(\calD\!\left(\Phi_{n,0}^0\right)  \right)-1\nonumber\\
&
=\dim(\calD(\Phi_\mu))+\left[\xeqref{k29}n(\dim (\calD (\Phi_\mu))-1)+3\right]-1\nonumber\\&=
(n+1)(\dim (\calD (\Phi_\mu))-1)+3.
\end{align}
Next, the fact that $\mu=\calR(\Phi_\mu)$, \eqref{k29}, \cref{m11b}, \cref{b03}, \cref{p11}, \eqref{k07}, and the definition of $\odot$ and $\boxdot$
show for all $t\in [0,T]$, $\theta\in\Theta$, $\ell\in [1,n-1]\cap\Z$ that
\begin{align} 
&\left[
\R^d\ni x\mapsto
\mu\!\left(X^{\theta,x}_{\ell}(\mathfrak{t}^{(\theta,n+1,k,\ell)}(\omega)t,\omega),
X^{(\theta,n+1,k,\ell),x}_{\ell}(\mathfrak{t}^{(\theta,n+1,k,\ell)}(\omega)t,\omega)
\right)\in\R^d\right]\nonumber
\\
&=\Bigl[
\R^d\ni x\mapsto
(\mu\circ \mathrm{Id}_{\R^{2d}})
\! \left(X^{\theta,x}_{\ell}(\mathfrak{t}^{(\theta,n+1,k,\ell)}(\omega)t,\omega),
X^{(\theta,n+1,k,\ell),x}_{\ell}(\mathfrak{t}^{(\theta,n+1,k,\ell)}(\omega)t,\omega)
\right)\in\R^d\Bigr]\nonumber
\\
&\in \calR\Bigl(\Bigl\{\Phi\in\bfN\colon
\calD(\Phi)=
\calD(\Phi_\mu)\odot
\mathfrak{n}^{2d}_{(n+1-\ell)(\dim (\calD (\Phi_\mu))-1)
-\dim (\calD (\Phi_\mu))+2}
\odot\left(\calD(\Phi_{\ell,0}^0)\boxdot\calD(\Phi_{\ell,0}^0)\right) \Bigr\}\Bigr)
\end{align}
and
\begin{align}
&
\dim \!\left(
\calD(\Phi_\mu)\odot
\mathfrak{n}^{2d}_{(n+1-\ell)(\dim (\calD (\Phi_\mu))-1)
-\dim (\calD (\Phi_\mu))+2}
\odot\left(\calD(\Phi_{\ell,0}^0)\boxdot\calD(\Phi_{\ell,0}^0)\right)\right)\nonumber\\
&=\dim\!\left( \calD(\Phi_\mu) \right)
+\dim\!\left(
\mathfrak{n}^{2d}_{(n+1-\ell)(\dim (\calD (\Phi_\mu))-1)
-\dim (\calD (\Phi_\mu))+2}\right) +\dim \!\left(\calD(\Phi_{\ell,0}^0)\boxdot\calD(\Phi_{\ell,0}^0)\right)-2\nonumber\\
&=\dim\!\left( \calD(\Phi_\mu) \right)
+\left[(n+1-\ell)(\dim (\calD (\Phi_\mu))-1)
-\dim (\calD (\Phi_\mu))+2\right] +
\dim \!\left(\calD(\Phi_{\ell,0}^0)\right)-2\nonumber
\\
&=
\dim\!\left( \calD(\Phi_\mu) \right)
+(n+1-\ell)(\dim (\calD (\Phi_\mu))-1)
-\dim (\calD (\Phi_\mu))+2\nonumber\\&\quad +
\left[\xeqref{k29}
\ell(\dim (\calD (\Phi_\mu))-1)+3\right]
-2\nonumber\\
&=
(n+1)(\dim (\calD (\Phi_\mu))-1)+3.
\end{align}
In addition, the fact that $\mu=\calR(\Phi_\mu)$, \eqref{k29}, \cref{m11b}, \cref{b03}, \cref{p11}, \eqref{k07}, and the definition of $\odot$ and $\boxdot$
show for all $t\in [0,T]$, $\theta\in\Theta$, $\ell\in [1,n]\cap\Z$ that
\begin{align} 
&\left[
\R^d\ni x\mapsto
\mu\bigl(X^{\theta,x}_{\ell-1}(\mathfrak{t}^{(\theta,n+1,k,\ell)}(\omega)t,\omega),
X^{(\theta,n+1,k,\ell),x}_{\ell-1}(\mathfrak{t}^{(\theta,n+1,k,\ell)}(\omega)t,\omega)
\bigr)\in\R^d\right]\nonumber\\
&=
\Bigl[
\R^d\ni x\mapsto
(
\mu\circ \mathrm{Id}_{\R^{2d}})\bigl(X^{\theta,x}_{\ell-1}(\mathfrak{t}^{(\theta,n+1,k,\ell)}(\omega)t,\omega),
X^{(\theta,n+1,k,\ell),x}_{\ell-1}(\mathfrak{t}^{(\theta,n+1,k,\ell)}(\omega)t,\omega)
\bigr)\in\R^d\Bigr]\nonumber\\
&\in \calR\Bigl(\Bigl\{\Phi\in\bfN\colon
\calD(\Phi)=
\calD(\Phi_\mu)\odot
\mathfrak{n}^{2d}_{(n+2-\ell)(\dim (\calD (\Phi_\mu))-1)
-\dim (\calD (\Phi_\mu))+2}
\odot\left(\calD(\Phi_{\ell-1,0}^0)\boxdot\calD(\Phi_{\ell-1,0}^0)\right) \Bigr\}\Bigr)
\end{align}
and
\begin{align} 
&
\dim\!\left(
\calD\!\left(
\Phi_\mu\odot
\mathfrak{n}^{2d}_{(n+2-\ell)(\dim (\calD (\Phi_\mu))-1)
-\dim (\calD (\Phi_\mu))+2}
\odot\left(\Phi_{\ell-1,0}^0\boxdot\Phi_{\ell-1,0}^0\right)\right) \right)\nonumber
\\
&=\dim\!\left( \calD(\Phi_\mu) \right)
+\dim\!\left(
\mathfrak{n}^{2d}_{(n+2-\ell)(\dim (\calD (\Phi_\mu))-1)
-\dim (\calD (\Phi_\mu))+2}\right) +\dim \!\left(\calD(\Phi_{\ell-1,0}^0)\boxdot\calD(\Phi_{\ell-1,0}^0)\right)-2\nonumber\\
&=\dim\!\left( \calD(\Phi_\mu) \right)
+\left[\xeqref{k07}
(n+2-\ell)(\dim (\calD (\Phi_\mu))-1)
-\dim (\calD (\Phi_\mu))+2
\right]
+\dim \!\left(\calD(\Phi_{\ell-1,0}^0)\right)-2\nonumber\\
&=\dim\!\left( \calD(\Phi_\mu) \right)
+\left[
(n+2-\ell)(\dim (\calD (\Phi_\mu))-1)
-\dim (\calD (\Phi_\mu))+2
\right]\nonumber\\&\quad 
+
\left[\xeqref{k29}
(\ell-1)(\dim (\calD (\Phi_\mu))-1)+3
\right]
-2\nonumber\\
&=
(n+1)(\dim (\calD (\Phi_\mu))-1)+3.
\label{k37}
\end{align}
Now,
\eqref{k30}--\eqref{k37} and \cref{b01b} show that there exists  $(\Phi_{n+1,t}^\theta)_{t\in[0,T],\theta\in \Theta}\subseteq\bfN$ such that for all $t\in[0,T]$, $\theta\in \Theta $, $x\in\R^d$ we have that
\begin{align}
&(\calR(\Phi_{n+1,t}^\theta))(x)\nonumber\\
&=
x+ W^\theta\left(\sup\!\left(\left\{\frac{kT}{m^{n+1}}\colon k\in\N_0\right\}\cap[0,t]\right),\omega\right)+t\mu(0,0)\nonumber\\
&\quad +\sum_{k=1}^{m}\frac{t
\mu\!\left(X^{\theta,x}_{n}(\mathfrak{t}^{(\theta,n+1,k,n)}(\omega)t,\omega),
X^{(\theta,n+1,k,n),x}_{n}(\mathfrak{t}^{(\theta,n+1,k,n)}(\omega)t,\omega)
\right)}{m}\nonumber\\
&
\quad +\sum_{\ell=1}^{n-1}
\sum_{k=1}^{m^{n+1-\ell}}
\frac{t\mu\bigl(X^{\theta,x}_{\ell}(\mathfrak{t}^{(\theta,n+1,k,\ell)}(\omega)t,\omega),
X^{(\theta,n+1,k,\ell),x}_{\ell}(\mathfrak{t}^{(\theta,n+1,k,\ell)}(\omega)t,\omega)
\bigr)}{m^{n+1-\ell}}\nonumber\\
&\quad 
-\sum_{\ell=1}^{n}\sum_{k=1}^{m^{n+1-\ell}}\frac{t\mu\bigl(X^{\theta,x}_{\ell-1}(\mathfrak{t}^{(\theta,n+1,k,\ell)}(\omega)t,\omega),
X^{(\theta,n+1,k,\ell),x}_{\ell-1}(\mathfrak{t}^{(\theta,n+1,k,\ell)}(\omega)t,\omega)
\bigr)}{m^{n+1-\ell}}\nonumber\\
&=
x+ W^\theta\left(\sup\!\left(\left\{\frac{kT}{m^{n+1}}\colon k\in\N_0\right\}\cap[0,t]\right),\omega\right)+t\mu(0,0)
\nonumber\\
&+ \sum_{\ell=1}^{n}
\sum_{k=1}^{m^{n+1-\ell}}\Biggl[\frac{t
\mu\!\left(X^{\theta,x}_{\ell}(\mathfrak{t}^{(\theta,n+1,k,\ell)} (\omega) t,\omega),
X^{(\theta,n+1,k,\ell),x}_{\ell}(\mathfrak{t}^{(\theta,n+1,k,\ell)}(\omega)t,\omega)
\right)
}{m^{n+1-\ell}}\nonumber\\
&\qquad\qquad\qquad-\frac{t\mu\!\left(X^{\theta,x}_{\ell-1}(\mathfrak{t}^{(\theta,n+1,k,\ell)}(\omega)t,\omega),
X^{(\theta,n+1,k,\ell),x}_{\ell-1}(\mathfrak{t}^{(\theta,n+1,k,\ell)}(\omega)t, \omega)
\right)}{m^{n+1-\ell}}\Biggr]\nonumber\\
&=X_{n+1}^{\theta,x}(t,\omega),
\label{k38}
\end{align}
\begin{align}
\dim (\calD (\Phi_{n+1,t}^\theta))=(n+1)(\dim (\calD (\Phi_\mu))-1)+3,\label{k39}
\end{align}
and
\begin{align}  
&
\calD(\Phi^\theta_{n+1,t})\nonumber\\&=\mathfrak{n}^{2d}_{(n+1)(\dim (\calD (\Phi_\mu))-1)+3}
\boxplus\left(\operatorname*{\boxplus}_{k=1}^m\calD(\Phi_\mu)\odot \left(\calD(\Phi^0_{n,t})\boxdot\calD(\Phi^0_{n,t})
\right)\right)\nonumber\\
&\quad \boxplus
\biggl(\operatorname*{\boxplus}_{\ell=1}^{n-1}
\operatorname*{\boxplus}_{k=1}^{m^{n+1-\ell}}
\calD(\Phi_\mu)\odot \mathfrak{n}^{2d}_{(n+1-\ell)(\dim (\calD (\Phi_\mu))-1)
-\dim (\calD (\Phi_\mu))+2}
\odot
\left(\calD(\Phi^0_{\ell,t})\boxdot\calD(\Phi^0_{\ell,t})
\right)
\biggr)\nonumber
\\
&
\quad \boxplus
\biggl(\operatorname*{\boxplus}_{\ell=1}^{n}
\operatorname*{\boxplus}_{k=1}^{m^{n+1-\ell}}
\calD(\Phi_\mu)\odot \mathfrak{n}^{2d}_{(n+2-\ell)(\dim (\calD (\Phi_\mu))-1)
-\dim (\calD (\Phi_\mu))+2}
\odot
\left(\calD(\Phi^0_{\ell-1,t})\boxdot\calD(\Phi^0_{\ell-1,t})
\right)
\biggr).\label{k40}
\end{align}
Hence, for all $t_1,t_2\in[0,T]$, $\theta_1,\theta_2\in \Theta$ we have that
\begin{align}
\calD(\Phi^{\theta_1}_{n+1,t_1})
=
\calD(\Phi^{\theta_2}_{n+1,t_2}).\label{k41}
\end{align}
Next, \cref{b04}, the definition of 
$\boxdot$, \eqref{k29}, and \eqref{k27} show for all $t\in[0,T]$ that
\begin{align} 
\supnorm{\calD(\Phi_\mu)\odot \left(\calD(\Phi^0_{n,t})\boxdot\calD(\Phi^0_{n,t})
\right)}&\leq \max\!\left\{
4d,\supnorm{\calD(\Phi_\mu)},
\supnorm{\calD(\Phi^0_{n,t})\boxdot\calD(\Phi^0_{n,t})}
\right\}\nonumber\\
&
\leq \max\!\left\{
4d,\supnorm{\calD(\Phi_\mu)},
2
\supnorm{\calD(\Phi^0_{n,t})}
\right\}\nonumber\\&\leq 2c(5m)^n.
\label{k42}
\end{align}
Moreover, \cref{b04}, 
\eqref{k07},
the definition of 
$\boxdot$, \eqref{k29}, and \eqref{k27} show for all $t\in[0,T]$, $\ell\in [1,n-1]\cap\Z$ that
\begin{align}
&
\supnorm{
\calD(\Phi_\mu)\odot \mathfrak{n}^{2d}_{(n+1-\ell)(\dim (\calD (\Phi_\mu))-1)
-\dim (\calD (\Phi_\mu))+2}
\odot
\left(\calD(\Phi^0_{\ell,t})\boxdot\calD(\Phi^0_{\ell,t})
\right)}\nonumber\\
&\leq 
\max\Bigl\{
4d,\supnorm{\calD(\Phi_\mu)},\supnorm{
\mathfrak{n}^{2d}_{(n+1-\ell)(\dim (\calD (\Phi_\mu))-1)
-\dim (\calD (\Phi_\mu))+2}},
\supnorm{\calD(\Phi^0_{\ell,t})\boxdot\calD(\Phi^0_{\ell,t})}
\Bigr\}\nonumber\\
&\leq 
\max\!\left\{
4d,\supnorm{\calD(\Phi_\mu)},
2
\supnorm{\calD(\Phi^0_{\ell,t})}
\right\}\leq 2c(5m)^\ell.
\end{align}
In addition, \cref{b04}, 
\eqref{k07},
the definition of 
$\boxdot$, \eqref{k29}, and \eqref{k27} show for all $t\in[0,T]$, $\ell\in [1,n]\cap\Z$ that
\begin{align}
 &
\supnorm{
\calD(\Phi_\mu)\odot \mathfrak{n}^{2d}_{(n+2-\ell)(\dim (\calD (\Phi_\mu))-1)
-\dim (\calD (\Phi_\mu))+2}
\odot
\left(\calD(\Phi^0_{\ell-1,t})\boxdot\calD(\Phi^0_{\ell-1,t})
\right)}\nonumber\\
&\leq 
\max\Bigl\{
4d,\supnorm{\calD(\Phi_\mu)},\supnorm{
\mathfrak{n}^{2d}_{(n+2-\ell)(\dim (\calD (\Phi_\mu))-1)
-\dim (\calD (\Phi_\mu))+2}},
\supnorm{\calD(\Phi^0_{\ell-1,t})\boxdot\calD(\Phi^0_{\ell-1,t})}
\Bigr\}\nonumber\\
&\leq 
\max\!\left\{
4d,\supnorm{\calD(\Phi_\mu)},
2
\supnorm{\calD(\Phi^0_{\ell-1,t})}
\right\}\leq 2c(5m)^{\ell-1}.
 \label{k44}
\end{align}
Next, \eqref{k40} and 
the triangle inequality show for all $\theta\in\Theta$,
$t\in [0,T]$ that
\begin{align} 
&
\supnorm{\calD(\Phi^\theta_{n+1,t})}\nonumber\\&\leq 
\supnorm{\mathfrak{n}^{2d}_{(n+1)(\dim (\calD (\Phi_\mu))-1)+3}}
+\left(\sum_{k=1}^m\supnorm{\calD(\Phi_\mu)\odot \left(\calD(\Phi^0_{n,t})\boxdot\calD(\Phi^0_{n,t})
\right)}\right)\nonumber\\
&\quad +
\Biggl(\sum_{\ell=1}^{n-1}
\sum_{k=1}^{m^{n+1-\ell}} \supnorm{
\calD(\Phi_\mu)\odot \mathfrak{n}^{2d}_{(n+1-\ell)(\dim (\calD (\Phi_\mu))-1)
-\dim (\calD (\Phi_\mu))+2}
\odot
\left(\calD(\Phi^0_{\ell,t})\boxdot\calD(\Phi^0_{\ell,t})
\right)}
\Biggr)\nonumber
\\
&
\quad +
\Biggl(\sum_{\ell=1}^{n}
\sum_{k=1}^{m^{n+1-\ell}}\supnorm{
\calD(\Phi_\mu)\odot \mathfrak{n}^{2d}_{(n+2-\ell)(\dim (\calD (\Phi_\mu))-1)
-\dim (\calD (\Phi_\mu))+2}
\odot
\left(\calD(\Phi^0_{\ell-1,t})\boxdot\calD(\Phi^0_{\ell-1,t})
\right)}
\Biggr).
\end{align}
This, \eqref{k07}, \eqref{k42}--\eqref{k44}, and  \eqref{k27} show for all $\theta\in\Theta$,
$t\in [0,T]$ that
\begin{align} \supnorm{\calD(\Phi^\theta_{n+1,t})}&\leq 
4d+\left(\sum_{k=1}^{m}2c(5m)^n\right)+\left(\sum_{\ell=1}^{n-1}\sum_{k=1}^{m^{n+1-\ell}}
2
c(5m)^\ell\right)+
\left(
\sum_{\ell=1}^{n}\sum_{k=1}^{m^{n+1-\ell}}
2c(5m)^{\ell-1}\right)\nonumber\\
&\leq 2c+2c\cdot 5^nm^{n+1}+
\left( \sum_{\ell=1}^{n-1}m^{n+1-\ell}2c\cdot 5^\ell m^\ell\right)
+\left(
\sum_{\ell=1}^{n}
m^{n+1-\ell}2c\cdot 5^{\ell-1} m^{\ell-1}\right)
\end{align} and
\begin{align}
  \supnorm{\calD(\Phi^\theta_{n+1,t})}
&\leq 
2cm^{n+1}\left(\sum_{\ell=0}^{n}5^\ell\right)+2cm^n\left(\sum_{\ell=1}^{n}5^{\ell-1}\right)\leq 2cm^{n+1}
\frac{5^{n+1}-1}{5-1}
+2cm^n\frac{5^{n}-1}{5-1}\nonumber\\&\leq c(5m)^{n+1}.
\end{align}
This, \eqref{k38}, \eqref{k39}, and \eqref{k41} complete the induction step. Induction hence completes the proof of \cref{k12}. 
\end{proof}
\subsection{DNN representation of Monte Carlo approximations}
In \cref{s14} below we consider the Monte Carlo approximation
$\frac{1}{K}\sum_{i=1}^{K}f(X^{i,x}_n(T))$ which approximates
$\E [f(X^{x}(T) )]$ where $X^x$ is the solution of the corresponding SDE. We will show that when $f$ is a DNN function then the Monte Carlo approximation
$\frac{1}{K}\sum_{i=1}^{K}f(X^{i,x}_n(T))$ is also a DNN function.
\begin{lemma}\label{s14}
Assume \cref{s12}, 
let $f\colon C(\R^d,\R)$, $\Phi_f\in \bfN$ satisfy that
$f=\mathcal{R}(\Phi_f)$, 
 and
 let $c\in \R$ satisfy that
\begin{align}
c\geq \max\{4d,\supnorm{\Phi_\mu},\supnorm{\Phi_f}\}.\label{k27b}
\end{align}
Then for all 
$n\in \N_0$, $K\in \N$ there exists $\Psi_{K,n}\in\bfN$
such that for all 
$x\in \R^d$ we have that
\begin{align}
\dim(\calD(\Psi_{K,n}))=
\dim(\calD(\Phi_f))+n(\dim (\calD (\Phi_\mu))-1)+2,
\end{align}
\begin{align}
\supnorm{\calD(\Psi_{K,n})}\leq Kc(5m)^n,
\end{align}
and
\begin{align} \begin{split}
 (\calR(\Psi_{K,n}))(x)= \frac{1}{K}\sum_{i=1}^{K}f(X^{i,x}_n(T)).
\end{split}\end{align}
\end{lemma}
\begin{proof}
[Proof of \cref{s14}]
The assumption of \cref{s14} and \cref{k12} show that for all $n\in \N_0$ there exists $(\Phi_{n,t}^\theta)_{t\in [0,T],\theta\in\Theta}\subseteq \bfN$ such that
 for all $t, t_1,t_2\in[0,T]$, $\theta,\theta_1,\theta_2\in \Theta$ we have that
\begin{align} \begin{split} 
&
\calD(\Phi_{n,t_1}^{\theta_1})=\calD(\Phi_{n,t_2}^{\theta_2}),\quad
\dim(\calD( \Phi_{n,t}^\theta ))= n(\dim (\calD (\Phi_\mu))-1)+3,\\
& 
 \supnorm{\calD( \Phi_{n,t}^\theta )} \leq c(5m)^n, \quad 
X^{\theta,x}_{n}(t,\omega)=(\calR( \Phi_{n,t}^\theta ))(x).\end{split}
\end{align}
This,  \cref{m11b}, \cref{b01b}, the definition of $\odot$, the triangle inequality (cf. \cref{b15}), \cref{b04}, and \eqref{k27b} show that 
for all $n\in \N_0$, $K\in \N$ there exists $\Psi_{K,n}\in \bfN$
such that for all 
$x\in\R^d$ we have that
\begin{align}
\calD(\Psi_{K,n})
=\operatorname*{\boxplus}_{i=1}^K\calD(\Phi_f)\odot
\calD(\Phi_{n,0}^0),
\end{align}
\begin{align}
\dim(\calD(\Psi_{K,n}))&=
\dim(\calD(\Phi_f))+\dim(\calD(\Phi_{n,0}^0))-1
\nonumber\\&
=\dim(\calD(\Phi_f))+n(\dim (\calD (\Phi_\mu))-1)+3-1\nonumber\\
&=\dim(\calD(\Phi_f))+n(\dim (\calD (\Phi_\mu))-1)+2,
\end{align}
\begin{align} 
\supnorm{\calD(\Psi_{K,n})}
&=\supnorm{\operatorname*{\boxplus}_{i=1}^K\calD(\Phi_f)\odot
\calD(\Phi_{n,0}^0)}\nonumber
\\&\leq K\supnorm{\calD(\Phi_f)\odot
\calD(\Phi_{n,0}^0)}\nonumber
\\&\leq K\max\{2d,\supnorm{\calD(\Phi_f)},\supnorm{\calD(\Phi_{n,0}^0)}\}\nonumber\\
&\leq K\max\{2d,\supnorm{\calD(\Phi_f)},c(5m)^n\}\nonumber\\&\leq Kc(5m)^n
\end{align}
and 
\begin{align}
(\calR(\Psi_{K,n}))(x)=  \frac{1}{K}\sum_{i=1}^{K}f(X^{i,x}_n(T)).
\end{align}
This completes the proof of \cref{s14}.
\end{proof}

\section{DNN approximations of the expectations}\label{s02}
In this section we provide the proof of   \cref{g01} and \cref{g01b}. The idea of the proof is the following.
First, we introduce an MLP setting (cf. \eqref{a05d}). Then we split the error into $3$ terms as in \eqref{x90} where the first term is estimated by an MLP error (cf. \cite[Theorem 3.1]{HKN2022}), the second term is estimated by a standard property of the variance, and the third term is estimated by the perturbation lemma, \cref{a01}.
\begin{theorem}\label{g01}
Assume \cref{m07}. Let $T\in(0,\infty)$,
$c\in [1,\infty)$,
$r\in\N$.  For every
$d\in \N$, $\varepsilon\in [0,1)$ let $\mu_{d,\varepsilon}\in C(\R^d\times \R^d,\R^d)$,
$f_{d,\varepsilon}\in C(\R^d,\R)$.
For every
$d\in \N$, $\varepsilon\in (0,1)$ let
$\Phi_{\mu_{d,\varepsilon}},\Phi_{f_{d,\varepsilon}}\in \bfN$
satisfy that
$\mu_{d,\varepsilon}=\calR(\Phi_{\mu_{d,\varepsilon}}) $ and
$f_{d,\varepsilon}=\calR(\Phi_{f_{d,\varepsilon}}) $.
Assume for all 
$d\in \N$,
$\varepsilon\in (0,1)$, $x,y,x_1,x_2,y_1,y_2\in \R^d$  that
\begin{align}\label{g02}
\left\lVert
\mu_{d,\varepsilon}(x_1,y_1)-
\mu_{d,\varepsilon}(x_2,y_2)\right\rVert\leq 0.5c\lVert x_1-y_1\rVert+0.5c \lVert x_2-y_2\rVert
\end{align}
\begin{align}\label{g03}
\left\lVert
\mu_{d,\varepsilon}(x_1,y_1)-
\mu_{d,0}(x_1,y_1)\right\rVert\leq cd^c\varepsilon+ 0.5cd^c\varepsilon\lVert x_1\rVert^r+
0.5cd^c\varepsilon\lVert y_1\rVert^r,
\end{align}
\begin{align}
\left\lvert
f_{d,\varepsilon}(x)-f_{d,0}(x)
\right\rvert\leq \varepsilon(1+ \lVert x\rVert^r),\quad \left \lvert f_{d,\varepsilon}( x )-f_{d,\varepsilon}(y)\right\rvert\leq c\lVert x-y\rVert,\label{g04b}
\end{align}
\begin{align}
\left\lvert
f_{d,\varepsilon}(0)
\right\rvert+
1+T
\left\lVert\mu_{d,\varepsilon}(0,0)\right\rVert
+\sqrt{T(d+4r)}\leq cd^c,\label{g05}
\end{align}
\begin{align}
\max \!\left\{\dim(\calD(\Phi_{f_{d,\varepsilon}})),\dim (\calD (\Phi_{\mu_{d,\varepsilon}}))\right\}\leq d^c\varepsilon^{-c},\label{k85}
\end{align} and
\begin{align}
\max\!\left\{\supnorm{\calD(\Phi_{\mu_{d,\varepsilon}})},
\supnorm{\calD(\Phi_{f_{d,\varepsilon}})}\right\}\leq d^c\varepsilon^{-c}.\label{k87}
\end{align}
Then the following items hold.
\begin{enumerate}[i)]
\item\label{i01} There exists a probability space $(\Omega,\mathcal {F}, \P)$, a family of standard Brownian motions $W^d\colon [0,T]\times \Omega\to \R^d$, $d\in \N$, and  $X=(X^{d,x}(t))_{d\in\N,x\in\R^d,t\in [0,T]}$ such that for every $d\in \N$, $ (X^{d,x}(t))_{t\in [0,T]} $ 
is a
$(\sigma( (W_s)_{s\in[0,t]} ))_{t\in[0,T]}$-adapted stochastic process with continuous sample paths
 and satisfies for all $t\in [0,T]$, $m\in [1,\infty)$ that
$
\E\!\left[
\sup_{s\in [0,T]}\left\lvert
X^{d,x}(s)\right\rvert^m\right]<\infty
$ and $\P$-a.s.\
\begin{align} \begin{split} X^{d,x}(t)=x+\int_{0}^{t}\E \!\left[\mu_{d,0}\bigl(y,X^{d,x}(s)\bigr)
\right]\bigr|_{y=X^{d,x}(s)}\,ds+W^{d}(t).\end{split}
\end{align}
\item\label{i02} There exist $(C_\delta)_{\delta\in (0,1)}\subseteq (0,\infty)$ and $(\Psi_{d,\epsilon}) _{d\in\N,\epsilon\in \N}\subseteq\bfN$ 
such that for all $d\in \N$, $\delta,\epsilon\in (0,1)$ we have that
$
\calP(\Psi_{d,\epsilon})\leq 96d^{3c}\left((2cd^c)^{r+1}e^{(r+2)cT}\right)^{3c+8+\delta}C_\delta\epsilon^{-(3c+8+\delta)}
$ and
\begin{align}
\left(
\int_{[0,1]^d}
\left\lvert
(\calR(\Psi_{d,\epsilon}))(x)
-\E\!\left[f_{d,0}(X^{d,x}(T))\right]\right\rvert^2 dx
\right)^\frac{1}{2}
<\epsilon.
\end{align}
\end{enumerate}

\end{theorem}

\begin{proof}
[Proof of \cref{g01}]
Let $(\Omega,\mathcal{F},\P)$ be a  probability space.
Let $\mathfrak{t}^\theta\colon\Omega\to[0,1]$, $\theta\in\Theta$, be i.i.d.\ random variables, assume for all $t\in [0,1] $ that
$\P(\mathfrak{t}^0\leq t)=t$.
For every $d\in \N$
let $W^{d,\theta}\colon [0,T]\times\Omega\to\R^d$, $\theta\in\Theta$, be i.i.d.\ standard 
Brownian motions with continuous sample paths.
Assume for every $d\in \N$ that 
$(\mathfrak{t}^\theta)_{ \theta\in\Theta}$ and
$(W^{d,\theta})_{ \theta\in\Theta}$ 
are independent. For every 
$d\in \N$, $\varepsilon\in (0,1)$,
$x\in \R^d$
let $X^{d,\theta,\varepsilon,x}_{n,m}\colon [0,T]\times\Omega\to\R^d$, $\theta\in\Theta$,  $n\in \N_0$, satisfy  for all
$\theta\in\Theta$, $m,n\in\N$, $t\in [0,T] $  that
$
X_{0,m}^{d,\theta,\varepsilon,x}(t)=0
$
and \begin{align}\label{a05d}  
\begin{split}
&X_{n,m}^{d,\theta,\varepsilon,x}(t)= x+ W^{d,\theta}\left(\sup\!\left(\left\{\frac{kT}{m^n}\colon k\in\N_0\right\}\cap[0,t]\right)\right)+t\mu_{d,\varepsilon}(0,0)\\&+ \sum_{\ell=1}^{n-1}
\sum_{k=1}^{m^{n-\ell}}\Biggl[\frac{t
\mu_{d,\varepsilon}\bigl(X^{d,\theta,\varepsilon,x}_{m,\ell}(\mathfrak{t}^{(\theta,n,k,\ell)}t),
X^{d,(\theta,n,k,\ell),\varepsilon,x}_{\ell}(\mathfrak{t}^{(\theta,n,k,\ell)}t)
\bigr)
}{m^{n-\ell}}
\\
&\qquad\qquad\qquad\qquad-
\frac{t
\mu_{d,\varepsilon}\bigl(X^{d,\theta,\varepsilon,x}_{\ell-1}(\mathfrak{t}^{(\theta,n,k,\ell)}t),
X^{d,(\theta,n,k,\ell),\varepsilon,x}_{\ell-1}(\mathfrak{t}^{(\theta,n,k,\ell)}t)
\bigr)}{m^{n-\ell}}\Biggr].
\end{split}\end{align}
First,
\cref{a01b} and \eqref{g02} imply that 
for every $d\in \N$, $\theta\in \Theta$, $\varepsilon\in [0,1)$, $x\in \R^d$
there exists a unique 
$(\sigma( (W^{d,\theta}(s))_{s\in[0,t]} ))_{t\in[0,T]}$-adapted 
stochastic process 
$(X^{d,\theta,\varepsilon,x}(t))_{t\in [0,T]}$
such that
for all $t\in [0,T]$
we have $\P$-a.s.\ that
\begin{align}  X^{d,\theta,\varepsilon,x}(t)=x+\int_{0}^{t}\E \!\left[\mu_{d,\varepsilon}\bigl(y,X^{d,\theta,\varepsilon,x}(s)\bigr)
\right]\bigr|_{y=X^{d,\theta,\varepsilon,x}(s)}\,ds+W^{d,\theta}(t)\label{k73}
\end{align}
and 
such that for all $m\in [1,\infty)$ we have that
\begin{align}
\E\!\left[
\sup_{s\in [0,T]}\left\lvert
X^{d,\theta,\varepsilon,x}(t)\right\rvert^m\right]<\infty.
\label{k74}
\end{align}
This proves \eqref{i01}.

Next,
\eqref{k73}, \eqref{k74},
\cref{a01} (applied for every $d\in\N$, $\theta\in \Theta$ with 
$T\gets T$, $c\gets c $, $b\gets cd^c$, $d\gets d$, $p\gets 2$, $r\gets r$,
$(\mu_\varepsilon)_{\varepsilon\in [0,1)}
\gets
(\mu_{d,\varepsilon})_{\varepsilon\in [0,1)}$,
$(f_\varepsilon)_{\varepsilon\in [0,1)}
\gets
(f_{d,\varepsilon})_{\varepsilon\in [0,1)}$,
$(\Omega,\mathcal{F},\P)\gets(\Omega,\mathcal{F},\P)$,
$(X^{\varepsilon,x}(t))_{x\in\R^d,\varepsilon\in[0,1), t\in [0,T]}\gets 
(X^{d,\theta,\varepsilon,x}(t))_{x\in\R^d,\varepsilon\in[0,1), t\in [0,T]}$
in the notation of \cref{a01}), and \eqref{g05} show
for all $d\in \N$, $\theta\in \Theta$, $\varepsilon\in (0,1)$, $x\in \R^d$, $t\in [0,T]$ that
\begin{align} \begin{split} 
\left\lVert X^{d,\theta,\varepsilon,x}(t)\right\rVert_{L^{2r}(\R;\R^d)}
&\leq \left(\lVert x\rVert
+T\lVert \mu_{d,\varepsilon}(0,0) \rVert+\sqrt{T(d+4r)}
\right)e^{cT}
\label{g70}\end{split}
\end{align}
and 
\begin{align}
&
\left\lVert f_{d,\varepsilon}(X^{d,\theta,\varepsilon,x}(t))
-
f_{d,0}(X^{d,\theta,0,x}(t))\right\rVert_{L^2(\P;\R)}\nonumber\\
&\leq 
cd^c\varepsilon \left(1+
\lVert x\rVert+T
\max\{
\left\lVert\mu_{d,\varepsilon}(0,0)\right\rVert,
\left\lVert\mu_{d,0}(0,0)\right\rVert
\}
+\sqrt{T(d+4r)}\right)^re^{(r+2)cT}\nonumber\\
&\leq\xeqref{g05}c d^c\varepsilon \left(
\lVert x\rVert+cd^c
\right)^r e^{(r+2)cT}
.\label{g74}
\end{align}
Moreover, the triangle inequality,
the fact that 
$\forall\,d\in\N, \theta\in\Theta,\varepsilon\in [0,1),t\in[0,T], x\in\R^d,n\in\N\colon \P\!\left( (X^{d,\theta,\varepsilon,x}_n(t),
X^{d,\theta,\varepsilon,x}(t))\in\cdot 
\right)=
\P\!\left( (X^{d,0,\varepsilon,x}_n(t),
X^{d,0,\varepsilon,x}(t))\in\cdot 
\right)
$, \eqref{g04b}, \cite[Theorem 3.1]{HKN2022}
(applied for every $d\in\N$, $\varepsilon\in [0,1)$, $x\in \R^d$ with
$T\gets T$, $L\gets c$, $d\gets d$, $\xi\gets x$, 
$\mu\gets \mu_{d,\varepsilon}$, $\Theta\gets\Theta$,
$\lVert\cdot \rVert\gets \lVert\cdot \rVert$,
$(\Omega,\mathcal{F},\P)\gets(\Omega,\mathcal{F},\P)$,
$(\mathfrak{u}^\theta)_{\theta\in\Theta}\gets (\mathfrak{t}^\theta)_{\theta\in\Theta}$,
$({W}^\theta)_{\theta\in\Theta}\gets ({W}^{d,\theta})_{\theta\in\Theta}$, 
$X\gets X^{d,0,\varepsilon,x}$,
$(X^\theta_{n,m})_{m\in\N,n\in\N_0,\theta\in\Theta}\gets (X^{d,\theta,\varepsilon,x}_{n,m})_{m\in\N,n\in\N_0,\theta\in\Theta}$ in the notation of \cite[Theorem 3.1]{HKN2022}),
\eqref{a05d}, \eqref{g02}, the fact that
$1+2cT\leq e^{2cT}$, and \eqref{g05}
show for all $d\in\N$, 
$\varepsilon\in [0,1)$, 
$t\in [0,T]$,
$x\in\R^d$,
$K,m,n\in\N$ that
\begin{align} 
\left\lVert
\frac{1}{K}
\sum_{i=1}^{K}\left[f_{d,\varepsilon}(X^{d,i,\varepsilon,x}_{n,m}(t))
-
f_{d,\varepsilon}(X^{d,i,\varepsilon,x}(t))\right]
\right\rVert_{L^2(\P;\R)}
&\leq \left\lVert f_{d,\varepsilon}(X^{d,0,\varepsilon,x}_{n,m}(t))
-f_{d,\varepsilon}(X^{d,0,\varepsilon,x}(t))
\right\rVert_{L^2(\P;\R)}\nonumber\\
&
\leq\xeqref{g04b} c
\left\lVert X^{d,0,\varepsilon,x}_{n,m}(t)
-X^{d,0,\varepsilon,x}(t)
\right\rVert_{L^2(\P;\R^d)}\nonumber
\\&\leq c \frac{e^{\frac{m}{2}}}{m^{\frac{n}{2}}}\left[\lVert x\rVert+\lVert \mu_{d,\varepsilon} (0,0) \rVert
t+\sqrt{Td}
\right]e^{cT}
 e^{2cTn}\nonumber\\
&\leq 
c \frac{e^{\frac{m}{2}}}{m^{\frac{n}{2}}}\left[\lVert x\rVert+\lVert \mu_{d,\varepsilon} (0,0) \rVert
t+\sqrt{Td}
\right]
 e^{3cTn}\nonumber\\&\leq
c \frac{e^{\frac{m}{2}}}{m^{\frac{n}{2}}}\left[\lVert x\rVert+cd^c
\right]
 e^{3cTn}
.
\label{g72}
\end{align}
Next, a standard variance estimate,
\eqref{g04b}, \eqref{g70}, the fact that $c\geq 1$,
and \eqref{g05}
imply for all $d\in\N$, $\theta\in\Theta$,
$\varepsilon\in [0,1)$, 
$t\in [0,T]$,
$x\in\R^d$,
$K\in\N$ that
\begin{align} 
&
\left\lVert
\frac{1}{K}
\sum_{i=1}^{K}
f_{d,\varepsilon}(X^{d,i,\varepsilon,x}(t))
-\E\!\left[f_{d,\varepsilon}(X^{d,0,\varepsilon,x}(t))\right]\right\rVert_{L^2(\P;\R)}\nonumber
\\&\leq \frac{\left\lVert
f_{d,\varepsilon}(X^{d,0,\varepsilon,x}(t))
\right\rVert_{L^2(\P;\R)}}{\sqrt{K}}\nonumber
\\
&\leq \xeqref{g04b}
 \frac{\left\lvert f_{d,\varepsilon}(0)\right\rvert+c\left\lVert
X^{d,0,\varepsilon,x}(t)
\right\rVert_{L^2(\P;\R^d)}}{\sqrt{K}}\nonumber
\\&\leq \xeqref{g70}
\frac{\left\lvert f_{d,\varepsilon}(0)\right\rvert+c\left[\lVert x\rVert +T\lVert \mu_\varepsilon(0,0)\rVert +\sqrt{T(d+4q)} \right]e^{cT}}{\sqrt{K}}\nonumber\\
&
\leq 
\frac{c\left[\left\lvert f_{d,\varepsilon}(0)\right\rvert+\lVert x\rVert +T\lVert \mu_\varepsilon(0,0)\rVert +\sqrt{T(d+4q)} \right]e^{cT}}{\sqrt{K}}\nonumber\\&\leq\xeqref{g05} \frac{c\left[\lVert x\rVert+cd^c\right]e^{cT}}{\sqrt{K}}.
\label{g73}
\end{align}
In addition, a telescoping sum argument shows for all
$d\in\N$, $\theta\in\Theta$,
$\varepsilon\in [0,1)$, 
$t\in [0,T]$,
$x\in\R^d$,
$m,n,K\in\N$
that
\begin{align} 
\frac{1}{K}
\sum_{i=1}^{K}f_{d,\varepsilon}(X^{d,i,\varepsilon,x}_{n,m}(t))
-\E\!\left[f_{d,0}(X^{d,0,0,x}(t))\right]
&
=
\frac{1}{K}
\sum_{i=1}^{K}\left[f_{d,\varepsilon}(X^{d,i,\varepsilon,x}_{n,m}(t))
-
f_{d,\varepsilon}(X^{d,i,\varepsilon,x}(t))\right]\nonumber\\
&\quad
+
\frac{1}{K}
\sum_{i=1}^{K}\left[
f_{d,\varepsilon}(X^{d,i,\varepsilon,x}(t))
-\E\!\left[f_{d,\varepsilon}(X^{d,0,\varepsilon,x}(t))\right]\right]\nonumber\\&\quad 
+
\E\!\left[f_{d,\varepsilon}(X^{d,0,\varepsilon,x}(t))\right]
-\E\!\left[f_{d,0}(X^{d,0,0,x}(t))\right].\label{x90}
\end{align}
This, the triangle inequality, \eqref{g72},  \eqref{g73}, and \eqref{g74}
imply for all
$d,m,n,K\in\N$, $\theta\in\Theta$,
$\varepsilon\in [0,1)$, 
$t\in [0,T]$
that
\begin{align}
&\left(\E \left[
\int_{[0,1]^d}
\left\lvert
\frac{1}{K}
\sum_{i=1}^{K}f_{d,\varepsilon}(X^{d,i,\varepsilon,x}_{n,m}(t))
-\E\!\left[f_{d,0}(X^{d,0,0,x}(t))\right]\right\rvert^2 dx 
\right]\right)^\frac{1}{2}\nonumber
\\
&=\left(
\int_{[0,1]^d}
\E \left[
\left\lvert
\frac{1}{K}
\sum_{i=1}^{K}f_{d,\varepsilon}(X^{d,i,\varepsilon,x}_{n,m}(t))
-\E\!\left[f_{d,0}(X^{d,0,0,x}(t))\right]\right\rvert^2 
\right]dx \right)^\frac{1}{2}\nonumber\\
&\leq\xeqref{x90}
\left(
\int_{[0,1]^d}
\E \left[
\left\lvert\frac{1}{K}
\sum_{i=1}^{K}\left[f_{d,\varepsilon}(X^{d,i,\varepsilon,x}_{n,m}(t))
-
f_{d,\varepsilon}(X^{d,i,\varepsilon,x}(t))\right]\right\rvert^2 
\right]dx \right)^\frac{1}{2}\nonumber\\
&\quad +
\left(
\int_{[0,1]^d}
\E \left[
\left\lvert\frac{1}{K}
\sum_{i=1}^{K}\left[
f_{d,\varepsilon}(X^{d,i,\varepsilon,x}(t))
-\E\!\left[f_{d,\varepsilon}(X^{d,0,\varepsilon,x}(t))\right]\right]\right\rvert^2 
\right]dx \right)^\frac{1}{2}\nonumber\\
&\quad +\left(
\int_{[0,1]^d}
\E \left[
\left\lvert\E\!\left[f_{d,\varepsilon}(X^{d,0,\varepsilon,x}(t))\right]
-\E\!\left[f_{d,0}(X^{d,0,0,x}(t))\right]\right\rvert^2 
\right]dx \right)^\frac{1}{2}\nonumber\\
&
\leq \Biggl(\int_{[0,1]^d}\Biggl[\xeqref{g72}c \frac{e^{\frac{m}{2}}}{m^{\frac{n}{2}}}\left[\lVert x\rVert+cd^c
\right]
 e^{3cTn}\Biggr]^2dx\Biggr)^{\frac{1}{2}}+ \Biggl(\int_{[0,1]^d}\Biggl[\xeqref{g73}\frac{c\left[\lVert x\rVert+cd^c\right]e^{cT}}{\sqrt{K}}
\Biggr]^2dx\Biggr)^{\frac{1}{2}}\nonumber\\
&\quad + \Biggl(\int_{[0,1]^d}\Biggl[\xeqref{g74}cd^c\varepsilon\left[\lVert x\rVert+cd^c\right]^r e^{(r+2)cT}\Biggr]^2dx\Biggr)^{\frac{1}{2}}\nonumber\\
&\leq \left(\int_{[0,1]^d}\left\lvert cd^c\left(\lVert x\rVert+cd^c\right)^r\right\rvert^2 dx\right)^\frac{1}{2}
\left(\frac{e^{\frac{m}{2}+3cTn}}{m^{\frac{n}{2}}}+\frac{e^{cT}}{\sqrt{K}}+\varepsilon
e^{(r+2)cT}
\right)
\nonumber\\
&\leq cd^c (2cd^c)^r
\left(\frac{e^{\frac{m}{2}+3cTn}}{m^{\frac{n}{2}}}+\frac{e^{cT}}{\sqrt{K}}+\varepsilon
e^{(r+2)cT}
\right)\nonumber\\
&= 2^r (cd^c)^{r+1}
\left(\frac{e^{\frac{m}{2}+3cTn}}{m^{\frac{n}{2}}}+\frac{e^{cT}}{\sqrt{K}}+\varepsilon
e^{(r+2)cT}
\right).
\label{a76}
\end{align}
For the next step for every 
$d\in \N$, $\epsilon\in (0,1)$ let $\varepsilon_{d,\epsilon}\in(0,1)$, 
$N_{d,\epsilon}\in\N$ satisfy that
\begin{align}
\varepsilon_{d,\epsilon}=\frac{\epsilon}{2^r (cd^c)^{r+1}e^{(r+2)cT}}\label{a60}
\end{align}
and
\begin{align}\label{a61}
N_{d,\epsilon}=\min\!\left\{n\in \N\cap[2,\infty)\colon2^r (cd^c)^{r+1}\frac{2e^{\frac{n}{2}+3cTn}}{n^\frac{n}{2}}\leq \frac{\epsilon}{2}\right\}.
\end{align}
Then \eqref{a76} proves
for all
$d\in\N$,
$\varepsilon\in [0,1)$
that
\begin{align} 
&\left(\E\!\left[\int_{[0,1]^d}
\left\lvert
\frac{1}{\lvert N_{d,\epsilon}\rvert^{N_{d,\epsilon}}}
\sum_{i=1}^{\lvert N_{d,\epsilon}\rvert^{N_{d,\epsilon}}}f_{d,\varepsilon_{d,\epsilon}}(X^{d,i,\varepsilon_{d,\epsilon},x}_{N_{d,\epsilon},N_{d,\epsilon}}(T))
-\E\!\left[f(X^{d,0,0,x}(T))\right]\right\rvert^2 dx\right]\right)^\frac{1}{2}\nonumber\\
&\leq \xeqref{a76}
2^r (cd^c)^{r+1}
\left(\frac{e^{\frac{N_{d,\epsilon}}{2}+3cTN_{d,\epsilon}}}{\lvert N_{d,\epsilon}\rvert^{\frac{N_{d,\epsilon}}{2}}}+\frac{e^{cT}}{\lvert N_{d,\epsilon}\rvert^{N_{d,\epsilon}}}+\varepsilon_{d,\epsilon}
e^{(r+2)cT}
\right)\nonumber\\
&\leq 
2^r (cd^c)^{r+1}
\left(\frac{2e^{\frac{N_{d,\epsilon}}{2}+3cTN_{d,\epsilon}}}{\lvert N_{d,\epsilon}\rvert^{\frac{N_{d,\epsilon}}{2}}}
+\varepsilon_{d,\epsilon}
e^{(r+2)cT}
\right)\nonumber\\
&\leq \xeqref{a60}\xeqref{a61}\frac{\epsilon}{2}+\frac{\epsilon}{2}=\epsilon.
\end{align}
Therefore, for all $d\in\N$,
$\varepsilon\in [0,1)$ there exists $\omega_{d,\epsilon}\in \Omega$ such that
\begin{align}
\int_{[0,1]^d}
\left\lvert
\frac{1}{\lvert N_{d,\epsilon}\rvert^{N_{d,\epsilon}}}
\sum_{i=1}^{\lvert N_{d,\epsilon}\rvert^{N_{d,\epsilon}}}f_{d,\varepsilon_{d,\epsilon}}(X^{d,i,\varepsilon_{d,\epsilon},x}_{N_{d,\epsilon},N_{d,\epsilon}}(T,\omega_{d,\epsilon}))
-\E\!\left[f(X^{d,0,0,x}(T))\right]\right\rvert^2 dx<\epsilon^2.\label{k82}
\end{align}
Next, 
\cref{s14} (applied for every
$d\in \N$, $\varepsilon\in (0,1)$
with
$T\gets T$, $m\gets N_{d,\epsilon}$, $d\gets d$,
$\mu\gets\mu_{d,\varepsilon}$,
$\Phi_\mu\gets\Phi_{\mu_{d,\varepsilon}}$,
$(\mathfrak{t}^\theta)_{\theta\in\Theta}\gets (\mathfrak{t}^\theta)_{\theta\in\Theta}$,
$(W^\theta)_{\theta\in\Theta}\gets(W^{d,\theta})_{\theta\in\Theta}$,
$(X^{\theta,x}_n)_{\theta\in\Theta,n\in \N_0,x\in\R^d}\gets
(X^{d,\theta,\varepsilon,x}_{n,N_{d,\epsilon}})_{\theta\in\Theta,n\in \N_0,x\in\R^d}
$, $\omega\gets\omega_{d,\epsilon}$,
$f\gets f_{d,\varepsilon}$,
$\Phi_f\gets\Phi_{f_{d,\varepsilon}}$,
$c\gets \max\!\left\{4d,\supnorm{\calD(\Phi_{\mu_{d,\varepsilon}})},
\supnorm{\calD(\Phi_{f_{d,\varepsilon}})}\right\}$,
$n\gets N_{d,\epsilon}$, 
$K\gets \lvert N_{d,\epsilon}\rvert^{N_{d,\epsilon}}$ in the notation of \cref{s14}) and
the fact that
$\forall\, d\in \N,\varepsilon\in (0,1)\colon \mu_{d,\varepsilon}=\calR(\Phi_{\mu_{d,\varepsilon}}) $ and
$f_{d,\varepsilon}=\calR(\Phi_{f_{d,\varepsilon}}) $
show that
for all 
 $d\in \N$, $\epsilon\in (0,1)$
there exists
$\Psi_{d,\epsilon}\in \bfN$
such that for all $x\in \R^d$ we have that
\begin{align}
\dim(\calD(\Psi_{d,\epsilon}))=
\dim(\calD(\Phi_{f_{d,\varepsilon_{d,\epsilon}}}))+N_{d,\epsilon}(\dim (\calD (\Phi_{\mu_{d,\varepsilon_{d,\epsilon}}}))-1)+2,
\end{align}
\begin{align} \begin{split} 
\supnorm{\calD(\Psi_{d,\epsilon})}\leq \lvert N_{d,\epsilon}\rvert^{N_{d,\epsilon}}\max\!\left\{4d,\supnorm{\calD(\Phi_{\mu_{d,\varepsilon_{d,\epsilon}}})},
\supnorm{\calD(\Phi_{f_{d,\varepsilon_{d,\epsilon}}})}\right\}(5N_{d,\epsilon})^{N_{d,\epsilon}} ,
\end{split}\end{align}
and
\begin{align}
 (\calR(\Psi_{d,\epsilon}))(x)=\frac{1}{\lvert N_{d,\epsilon}\rvert^{N_{d,\epsilon}}}
\sum_{i=1}^{\lvert N_{d,\epsilon}\rvert^{N_{d,\epsilon}}}f_{d,\varepsilon_{d,\epsilon}}(X^{d,i,\varepsilon_{d,\epsilon},x}_{N_{d,\epsilon},N_{d,\epsilon}}(T, \omega_{d,\epsilon})).\label{k01}
\end{align}
This, \eqref{k82}, \eqref{k85}, and \eqref{k87} imply for all $d\in \N$, $\epsilon\in (0,1)$ that
\begin{align}
\int_{[0,1]^d}
\left\lvert
(\calR(\Psi_{d,\epsilon}))(x)
-\E\!\left[f(X^{d,0,0,x}(T))\right]\right\rvert^2 dx<\epsilon^2,
\end{align}
\begin{align} 
&
\dim(\calD(\Psi_{d,\epsilon}))=
\dim(\calD(\Phi_{f_{d,\varepsilon_{d,\epsilon}}}))+n(\dim (\calD (\Phi_{\mu_{d,\varepsilon_{d,\epsilon}}}))-1)+2\nonumber\\
&
\leq 3N_{d,\epsilon}\max \!\left\{\dim(\calD(\Phi_{f_{d,\varepsilon_{d,\epsilon}}})),\dim (\calD (\Phi_{\mu_{d,\varepsilon_{d,\epsilon}}}))\right\}\nonumber\\
&\leq \xeqref{k85}3N_{d,\epsilon}d^c\lvert \varepsilon_{d,\epsilon}\rvert^{-c},
\label{k84}
\end{align}
and
\begin{align} 
\supnorm{\calD(\Psi_{d,\epsilon})}&\leq \lvert N_{d,\epsilon}\rvert^{N_{d,\epsilon}}\max\!\left\{4d,\supnorm{\calD(\Phi_{\mu_{d,\varepsilon_{d,\epsilon}}})},
\supnorm{\calD(\Phi_{f_{d,\varepsilon_{d,\epsilon}}})}\right\}(5N_{d,\epsilon})^{N_{d,\epsilon}}\nonumber \\
&\leq\xeqref{k87} 5^{N_{d,\epsilon}} \lvert N_{d,\epsilon}\rvert^{2N_{d,\epsilon}}4d^c\lvert\varepsilon_{d,\epsilon}\rvert^{-c}.
\label{k86}\end{align}
This, the fact that $\forall\, \Phi\in\bfN\colon \calP(\Phi)\leq 2\dim (\calD(\Phi))\supnorm{\calD(\Phi)}^2$, 
and \eqref{a60} prove for all $d\in \N$, $\epsilon\in (0,1)$ that
\begin{align} 
\calP(\Psi_{d,\epsilon})&\leq 2
\dim(\calD(\Psi_{d,\epsilon}))
\supnorm{\calD(\Psi_{d,\epsilon})}^2\nonumber\\
&
=2\left(\xeqref{k84}3N_{d,\epsilon}d^c\lvert \varepsilon_{d,\epsilon}\rvert^{-c}\right)\left(\xeqref{k86} 5^{N_{d,\epsilon}} \lvert N_{d,\epsilon}\rvert^{2N_{d,\epsilon}}4d^c\lvert\varepsilon_{d,\epsilon}\rvert^{-c}\right)^2\nonumber\\
&=96 N_{d,\epsilon}5^{2N_{d,\epsilon}}
\lvert N_{d,\epsilon}\rvert^{4N_{d,\epsilon}}
d^{3c}
\lvert\varepsilon_{d,\epsilon}\rvert^{-3c}\nonumber\\
&=
96 N_{d,\epsilon}5^{2N_{d,\epsilon}}
\lvert N_{d,\epsilon}\rvert^{4N_{d,\epsilon}}
d^{3c}\left(\xeqref{a60}\frac{\epsilon}{2^r (cd^c)^{r+1}e^{(r+2)cT}}\right)^{-3c}\nonumber\\
&=
96 N_{d,\epsilon}5^{2N_{d,\epsilon}}
\lvert N_{d,\epsilon}\rvert^{4N_{d,\epsilon}}
d^{3c}\left(2^r (cd^c)^{r+1}e^{(r+2)cT}\right)^{3c}
\epsilon^{-3c}
\label{k88}
\end{align}
For the next step let $C_\delta\in \R$, $\delta\in (0,1)$, satisfy for all $\delta\in (0,1)$ that
\begin{align}
C_\delta=\sup_{n\in [2,\infty)\cap\Z}\left[5^{2n}nn^{4n} \left(\frac{2e^{\frac{n-1}{2} +3cT(n-1)}}{(n-1)^\frac{n-1}{2}}\right)^{8+\delta}
\right].
\end{align}
Note  for all $\delta\in (0,1)$ that
\begin{align} 
C_\delta&=
\sup_{n\in [2,\infty)\cap\Z}\left[
5^{2n}n^5n^{4n-4}
\frac{\left(2e^{\frac{n-1}{2} +3cT(n-1)}\right)^{8+\delta}}{(n-1)^{4n-4}(n-1)^{\frac{\delta(n-1)}{2}}}\right]\nonumber\\
&=
\sup_{n\in [2,\infty)\cap\Z}\left[
5^{2n}n^5\left(\frac{n}{n-1}\right)^{4n-4}
\frac{\left(2e^{\frac{n-1}{2} +3cT(n-1)}\right)^{8+\delta}}{(n-1)^{\frac{\delta(n-1)}{2}}}\right]\nonumber\\
&\leq 
\sup_{n\in [2,\infty)\cap\Z}\left[
5^{2n}n^52^{4n-4}
\frac{\left(2e^{\frac{n-1}{2} +3cT(n-1)}\right)^{8+\delta}}{(n-1)^{\frac{\delta(n-1)}{2}}}\right]<\infty.
\label{k90}
\end{align}
Furthermore, \eqref{a61} shows for all $d\in \N$, $\epsilon\in (0,1)$ that
\begin{align}
\epsilon\leq\xeqref{a61} (2cd^c)^{r+1} \frac{2e^{\frac{N_{d,\epsilon}-1}{2}+3cT(N_{d,\epsilon}-1)}}{(N_{d,\epsilon}-1)^\frac{N_{d,\epsilon}-1}{2}}.
\label{k89}
\end{align}
This, \eqref{k88}, and \eqref{k90} show for all
$d\in \N$, $\delta,\epsilon\in (0,1)$ that
\begin{align} 
&
\calP(\Psi_{d,\epsilon})\epsilon^{3c+8+\delta}
\nonumber\\&=\xeqref{k88}
96 N_{d,\epsilon}5^{2N_{d,\epsilon}}
\lvert N_{d,\epsilon}\rvert^{4N_{d,\epsilon}}
d^{3c}\left(2^r (cd^c)^{r+1}e^{(r+2)cT}\right)^{3c}
\epsilon^{-3c}\cdot \epsilon^{3c+8+\delta}\nonumber\\
&=96 N_{d,\epsilon}5^{2N_{d,\epsilon}}
\lvert N_{d,\epsilon}\rvert^{4N_{d,\epsilon}}
d^{3c}\left(2^r (cd^c)^{r+1} e^{(r+2)cT}\right)^{3c}
 \epsilon^{8+\delta}\nonumber\\
&\leq 
96 N_{d,\epsilon}5^{2N_{d,\epsilon}}
\lvert N_{d,\epsilon}\rvert^{4N_{d,\epsilon}}
d^{3c}\left(2^r (cd^c)^{r+1} e^{(r+2)cT}\right)^{3c}
 \left(\xeqref{k89}(2cd^c)^{r+1} \frac{2e^{\frac{N_{d,\epsilon}-1}{2}+3cT(N_{d,\epsilon}-1)}}{(N_{d,\epsilon}-1)^\frac{N_{d,\epsilon}-1}{2}}\right)^{8+\delta}\nonumber\\
&\leq 96d^{3c}\left((2cd^c)^{r+1}e^{(r+2)cT}\right)^{3c+8+\delta} \left[
N_{d,\epsilon}5^{2N_{d,\epsilon}}
\lvert N_{d,\epsilon}\rvert^{4N_{d,\epsilon}}
\left( \frac{2e^{\frac{N_{d,\epsilon}-1}{2}+3cT(N_{d,\epsilon}-1)}}{(N_{d,\epsilon}-1)^\frac{N_{d,\epsilon}-1}{2}}\right)^{8+\delta}
\right]\nonumber\\
&\leq 96d^{3c}\left((2cd^c)^{r+1}e^{(r+2)cT}\right)^{3c+8+\delta}C_\delta<\infty
\end{align}
and hence
$
\calP(\Psi_{d,\epsilon})\leq 96d^{3c}\left((2cd^c)^{r+1}e^{(r+2)cT}\right)^{3c+8+\delta}C_\delta\epsilon^{-(3c+8+\delta)}
$.
This completes the proof of \cref{g01}.
\end{proof}
Finally, we provide the proof of \cref{g01b}.
\begin{proof}[Proof of \cref{g01b}]
First,
the fact that $\forall\,\Phi\in \bfN\colon \max \{\supnorm{\calD(\Phi)}, \dim (\calD(\Phi))\}\leq \calP(\Phi)$ and
\eqref{g09} show for all $d\in \N$, $\varepsilon\in (0,1)$ that
\begin{align}
\max \!\left\{\dim(\calD(\Phi_{f_{d,\varepsilon}})),\dim (\calD (\Phi_{\mu_{d,\varepsilon}}))\right\}\leq
\max\!\left\{\calP(\Phi_{\mu_{d,\varepsilon}}),
\calP(\Phi_{f_{d,\varepsilon}})\right\}\leq 
 d^c\varepsilon^{-c},\label{k85b}
\end{align} and
\begin{align}
\max\!\left\{\supnorm{\calD(\Phi_{\mu_{d,\varepsilon}})},
\supnorm{\calD(\Phi_{f_{d,\varepsilon}})}\right\}\leq
\max\!\left\{\calP(\Phi_{\mu_{d,\varepsilon}}),
\calP(\Phi_{f_{d,\varepsilon}})\right\}\leq  d^c\varepsilon^{-c}.\label{k87b}
\end{align}
Next,
\eqref{g08} implies that there exist $\mathbf{c}\in [c,\infty)$ such that for all $d\in \N$, $\varepsilon\in (0,\infty)$ we have that
\begin{align}
\left\lvert 
f_{d,\varepsilon}(0)
\right\rvert+
1+T
\left\lVert\mu_{d,\varepsilon}(0,0)\right\rVert
+\sqrt{T(d+4r)}\leq \mathbf{c}d^\mathbf{c}.\label{g05b}
\end{align}
Now, \eqref{g04}--\eqref{g07}, \eqref{k85b}--\eqref{g05b}, and \cref{g01} (with $c\gets \mathbf{c}$ in the notation of \cref{g01}) complete the proof of \cref{g01b}. 
\end{proof}

\section{Conclusion}
With Theorem~\ref{g01b}, we mathematically prove that neural networks can overcome the curse of dimensionality when approximating McKean-Vlasov SDEs. This hence can be seen as a motivation to develop in the future DNN based algorithms which can approximately solve McKean-Vlasov SDEs, together with their full error and complexity analysis, for which we hopefully will be able to demonstrate that the corresponding complexity of the algorithm only grows polynomially in the dimension $d$ and the reciprocal of the accuracy $\epsilon$. We believe that our theoretical result is a first step toward that goal. We refer to \cite{GMW2022,HHL2023,PW2022} for existing deep algorithms for McKean-Vlasov SDEs, but without a full convergence and complexity analysis. Moreover, for future research, it would be interesting to analyze if one can extend our Theorem~\ref{g01b} also for forward backward McKean--Vlasov SDEs.

{
\bibliographystyle{acm}
\bibliography{References}
}

\end{document}